\documentclass[11pt,reqno]{amsart}
\usepackage[T1]{fontenc}
\usepackage[utf8]{inputenc}
\usepackage{lmodern,amsmath,amssymb,amsthm,mathtools,mathrsfs}
\usepackage[margin=1in]{geometry}
\usepackage{microtype}
\usepackage{graphicx,tikz,placeins}
\graphicspath{{Figures/}}
\usepackage[hidelinks,bookmarksnumbered]{hyperref}
\allowdisplaybreaks[2]
\numberwithin{equation}{section}
\newtheorem{theorem}{Theorem}[section]
\newtheorem{proposition}[theorem]{Proposition}
\newtheorem{lemma}[theorem]{Lemma}
\newtheorem{corollary}[theorem]{Corollary}
\theoremstyle{definition}
\newtheorem{example}[theorem]{Example}
\theoremstyle{remark}
\newtheorem{remark}[theorem]{Remark}
\newcommand{\C}{\mathbb C}
\newcommand{\R}{\mathbb R}
\newcommand{\N}{\mathbb N}
\newcommand{\Sh}{\mathcal K}
\newcommand{\AllSh}{\mathcal S}
\newcommand{\dd}{\,\mathrm d}
\newcommand{\eps}{\varepsilon}
\newcommand{\Ree}{\operatorname{Re}}
\newcommand{\supp}{\operatorname{supp}}
\newcommand{\conv}{\operatorname{conv}}
\newcommand{\dist}{\operatorname{dist}}
\newcommand{\divz}{\operatorname{div}_0}
\newcommand{\Area}{\operatorname{Area}}
\newcommand{\Hull}{\mathsf H}
\newcommand{\ShadowPot}{\mathcal V}
\newcommand{\centroid}{\mathfrak c}
\newcommand{\Blev}{\mathcal B}

\newcommand{\rootmu}{\mu^{\mathrm{root}}}
\newcommand{\rootC}{\mathscr C}
\newcommand{\LogDer}{\Lambda}
\DeclareMathOperator{\acc}{acc}

\hypersetup{ pdftitle={The full shadow of a polynomial I. The limiting core and the transient set}, pdfauthor={Christian Hagg and Boris Shapiro}, pdfsubject={Polynomial derivative zeros and their limiting supports}, pdfkeywords={Gauss--Lucas theorem, derivative zeros, Hausdorff convergence, logarithmic potentials}}
\title[The full shadow of a polynomial I] {The full shadow of a polynomial I. The limiting core and the transient set}
\author{Christian H\"agg}
\address{Department of Mathematics, Stockholm University, SE--106 91 Stockholm, Sweden}
\email{hagg@math.su.se}
\author{Boris Shapiro}
\address{Department of Mathematics, Stockholm University, SE--106 91 Stockholm, Sweden}
\email{shapiro@math.su.se}
\date{September 15, 2026}
\begin{document}
\begin{abstract}
The full shadow of a complex polynomial \(P\) of degree \(d\ge2\) is the closure of all zeros of \((P^n)^{(m)}\), \(n\ge1\), \(0\le m<dn\). We prove that it consists of a compact connected limiting core containing all roots and critical points of \(P\), together with isolated transient zeros, each occurring at only finitely many powers. The core is the Hausdorff limit of the zero sets pooled over all derivative orders as \(n\to\infty\) and the support of their limiting normalized root-counting measure. For \(m/n\to\alpha\in(0,d)\), we establish weak convergence of normalized root-counting measures and Hausdorff convergence of zero sets to the limiting supports, both stable under coefficient perturbations. We exhibit quartics with infinitely many transients. A bounded operator constructed from polynomial differentiators realizes the full shadow, core, and transients as its spectrum, Fredholm essential spectrum, and discrete spectrum, respectively.
\end{abstract}
\maketitle
\setcounter{tocdepth}{2}
\begingroup\small\tableofcontents\endgroup

\section{Introduction}\label{sec:intro}
The Gauss--Lucas theorem places every zero of each nonzero derivative \((P^n)^{(m)}\) of a positive integer power \(P^n\) of a polynomial \(P\) in the convex hull \(\conv Z(P)\) of its roots; see, for example, \cite[Theorem~(6,1), p.~22, and Exercise~2, p.~24]{Marden}. (Unless stated otherwise, we assume that \(P\) is a monic polynomial of degree \(d\ge2\).)

This paper contains the initial results of the study of the smallest closed set \(\AllSh_P\) containing all such zeros. (For further information see \cite{Shadow 2}.) The set \(\AllSh_P\) carries a natural limiting root-counting measure whose closed support can be strictly smaller than \(\AllSh_P\). The resulting decomposition is given by
\begin{equation}\label{eq:split}
 \AllSh_P=\Sh_P\cup T_P,
\end{equation}
where \(\Sh_P\) is the closed support of this limiting measure and every point of \(T_P\) is isolated in \(\AllSh_P\). We will show that the set \(T_P\) is locally finite in \(\C\setminus\Sh_P\) and may accumulate only on \(\Sh_P\).

Figure~\ref{fig:shadows} shows, for a cubic and a quintic polynomial, all zeros of all derivatives of a certain high power \(P^n\). By Theorem~\ref{thm:intro-sets} below, such complete zero sets approximate \(\Sh_P\) as the power \(n\) grows.

\begin{figure}[hb]
\centering
\includegraphics[height=0.37\textwidth]{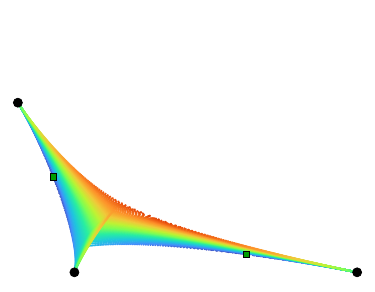}\hfill
\includegraphics[height=0.37\textwidth]{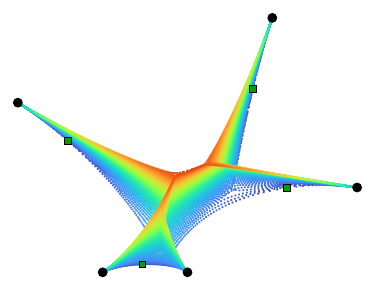}
\caption{All zeros of all derivatives \((P^n)^{(m)}\), \(0<m<n\deg P\). Left: \(P(z)=(z-1)(z-6)(z-3i)\) and \(n=80\), about \(19\,000\) zeros. Right: \(P(z)=z(z-1)(z-3-i)(z-2-3i)(z+1-2i)\) and \(n=50\), about \(25\,000\) zeros. A zero of \((P^n)^{(m)}\) is colored according to the ratio \(m/n\in(0,\deg P)\), from blue for small ratios to red for ratios close to \(\deg P\). Black disks: the roots of \(P\); green squares: the critical points of \(P\).}
\label{fig:shadows}
\end{figure}

\subsection{The full shadow of a polynomial}
Obviously any nonzero polynomial may be normalized to be monic without changing the objects studied below. For a nonzero polynomial \(H\), let \(Z(H)\) denote its set of distinct zeros and let \(\operatorname{ord}_a(H)\) be the multiplicity of \(H\) at \(a\). With \(\delta_a\) denoting unit mass at \(a\), we identify the positive zero divisor of \(H\) with the associated \emph{atomic measure}
\[
 \divz H:=\sum_{a\in Z(H)}\operatorname{ord}_a(H)\,\delta_a.
\]
For nonconstant \(H\), the measure
\[
\mu_H^{\mathrm{root}}=\frac{1}{\deg H}\divz H
\]
is called the \emph{normalized root-counting measure} of \(H\).

\medskip
Write \(\N=\{1,2,\ldots\}\) and set
\begin{equation}\label{eq:descendants}
 R_{P;m,n}(z):=(P^n)^{(m)}(z),\qquad n\in\N,\quad 0\le m<dn.
\end{equation}
Following \cite{BHS}, we call these polynomials the \emph{Rodrigues descendants} of \(P\). We include in their list derivative order zero: \(R_{P;0,n}=P^n\), whose distinct zero set is \(Z(P)\). The first excluded order is \(m=dn\), which gives a nonzero constant and hence no zeros; larger orders give the zero polynomial and are also excluded. Thus the ratio \(m/n\) ranges over \([0,d)\).

For a positive integer \(n\), the \(n\)-th \emph{complete zero set} is defined by
\begin{equation}\label{eq:pooled-definitions}
 \mathcal Z_{P,n}:=\bigcup_{m=0}^{dn-1}Z(R_{P;m,n}).
\end{equation}
A \emph{finite-level zero} of \(P\) is a point of some \(\mathcal Z_{P,n}\). The \emph{full shadow of \(P\)} is
\begin{equation}\label{eq:all-shadow}
 \AllSh_P:=\overline{\bigcup_{n\ge1}\mathcal Z_{P,n}}.
\end{equation}
Note that the union of all zeros before taking the closure is countable. Its finite \(M\)-th truncation is
\[
 \mathcal F_{P,M}:=\bigcup_{1\le n\le M}\mathcal Z_{P,n}.
\]
Repeated application of Gauss--Lucas gives \(\mathcal Z_{P,n}\subseteq\Hull_P:=\conv Z(P)\); in particular, \(\AllSh_P\) is compact.

(For completeness, for \(P\) being a nonzero constant we set its full shadow to be empty, and for a polynomial \(P\) with one distinct root its full shadow coincides with this root).

\medskip
The degree of the Rodrigues descendant \(R_{P;m,n}\) equals \(dn-m\). Its normalized root-counting measure is then given by
\begin{equation}\label{eq:zero-measure}
 \mu_{P;m,n}^{\mathrm{root}}:=\frac{1}{dn-m}\divz R_{P;m,n}.
\end{equation}
Collecting together derivatives of all orders, with multiplicities, gives the root-counting measure
\begin{equation}\label{eq:pooled-measure}
 \begin{gathered}
 \overline\mu_{P,n}^{\mathrm{root}}:=\frac1{\mathcal N_{d,n}}
       \sum_{m=0}^{dn-1}\divz R_{P;m,n},\\
 \mathcal N_{d,n}:=\sum_{m=0}^{dn-1}(dn-m)=\frac{dn(dn+1)}2.
 \end{gathered}
\end{equation}
Thus the set \(\mathcal Z_{P,n}\) records locations of roots, whereas the measure \(\overline\mu_{P,n}^{\mathrm{root}}\) also records their multiplicities.

\medskip
We write \(\Longrightarrow\) for weak convergence of finite measures and \(\supp\nu\) for the closed support of a measure \(\nu\). For nonempty compact sets, we use the Hausdorff distance
\[
 d_{\mathrm H}(A,B):=\max\left\{
  \sup_{a\in A}\dist(a,B),\ \sup_{b\in B}\dist(b,A)\right\}.
\]
Convergence of polynomials denoted by \(\to\) always means coefficient convergence in a fixed degree unless another convention is explicitly stated. We use \(\longrightarrow\) with its usual contextual meaning; when its arguments are compact sets, it denotes convergence in Hausdorff distance.

\medskip
Let us now formulate four main results of this paper.

\subsection{The limiting measures and the limiting core}
Our first result is a substantial strengthening of the main result of \cite{BHS}. In particular, it concerns both every zero of the respective Rodrigues descendants and their limiting distribution. For a compactly supported measure \(\mu\) on \(\C\), let \(\supp\mu\) denote its closed support.

\begin{theorem}[Fixed-ratio convergence]\label{thm:intro-fixed-ratio}
For every monic polynomial \(P\) of degree \(d\ge2\) and every \(\alpha\in(0,d)\), there is a probability measure \(\mu_{\alpha,P}\) supported in \(\Hull_P\) such that for any sequence \(\{m_n\}\) satisfying \(0\le m_n<dn\) with \(m_n/n\to\alpha\)

 \smallskip
\noindent
{\rm(i)}
\[
 \mu_{P;m_n,n}^{\mathrm{root}}\Longrightarrow\mu_{\alpha,P},\qquad
 Z(R_{P;m_n,n})\longrightarrow\supp\mu_{\alpha,P}
 \quad\text{in Hausdorff distance}.
\]

\smallskip
\noindent
{\rm(ii)} Moreover, if \(\{P_n\}\) is a sequence of monic polynomials of degree \(d\) converging coefficientwise to \(P\), i.e. \(P_n\to P\), then
\[
 \mu_{P_n;m_n,n}^{\mathrm{root}}\Longrightarrow\mu_{\alpha,P},\qquad
 Z(R_{P_n;m_n,n})\longrightarrow\supp\mu_{\alpha,P}
 \quad\text{in Hausdorff distance}.
\]
\medskip

\smallskip
\noindent
{\rm(iii)} The measure \(\mu_{\alpha,P}\) depends only on \((P,\alpha)\). It is jointly weakly continuous in these parameters, and its support is jointly Hausdorff-continuous.

\smallskip
\noindent
{\rm(iv)} At fixed degree, convergence of test-function integrals is uniform on compact coefficient families and compact subintervals of \((0,d)\); the corresponding support convergence is uniform in Hausdorff distance.
\end{theorem}

The measure \(\mu_{\alpha,P}\) was first described in \cite{BHS} using the same notation. Under the assumption that \(P\) and \(P'\) have simple roots, \cite[Theorems~1.1 and~1.7]{BHS} obtained both the algebraic equation satisfied almost everywhere by its Cauchy transform and a logarithmic-potential formula in terms of the tropical trace of the saddle landscape. Here we remove that simplicity assumption and give the following direct min--max barrier characterization of the potential. Section~\ref{sec:limits} proves
\begin{equation}\label{eq:fixed-potential-preview}
 U_{\mu_{\alpha,P}}(z)=
 \frac{\Blev_P(z,\alpha)-\mathfrak h_d(\alpha)}{d-\alpha},
 \qquad
 \mathfrak h_d(\alpha)=d\log d-\alpha\log\alpha
 -(d-\alpha)\log(d-\alpha),
\end{equation}
where \(\Blev_P(z,\alpha)\) is the min--max barrier separating \(z\) from infinity for the logarithmic landscape \(\log|P(w)|-\alpha\log|w-z|\). Thus \(\mu_{\alpha, P}\) is concretely determined by an explicit min--max barrier.

\medskip
Define the \emph{limiting core} of \(P\) by
\begin{equation}\label{eq:core-union}
 \Sh_P:=\overline{\bigcup_{0<\alpha<d}\supp \mu_{\alpha,P}}.
\end{equation}
Observe that \(\Sh_P\subseteq \AllSh_P\). The \emph{shadow measure} \(\rho_P\) supported on \(\Sh_P\) is the weighted average
\begin{equation}\label{eq:rho-intro}
 \rho_P:=\frac2{d^2}\int_0^d(d-\alpha)\mu_{\alpha,P}\dd\alpha.
\end{equation}

\begin{remark}
The above integral is understood in the weak sense: its integral against a continuous test function is the corresponding scalar integral of the test-function integrals. The weight reflects the degree \(dn-m\) of each descendant. For constants and polynomials with one distinct root, the core is, respectively, empty or the root singleton.
\end{remark} 

\begin{theorem}[All orders and coefficient stability]\label{thm:intro-sets}
For any sequence \(\{P_n\}\) of monic degree \(d\) polynomials such that \(P_n\to P\) coefficientwise,

\smallskip
\noindent
{\rm (i)}
\[
 \overline\mu_{P_n,n}^{\mathrm{root}}\Longrightarrow\rho_P,
 \qquad \supp\rho_P=\Sh_P,
\]
and
\begin{equation}\label{eq:two-set-limits-intro}
 \mathcal Z_{P_n,n}\longrightarrow\Sh_P,
 \qquad \mathcal F_{P_n,n}\longrightarrow\AllSh_P
 \quad\text{in Hausdorff distance}.
\end{equation}

\smallskip
\noindent
{\rm (ii)}
The maps \(P\mapsto\rho_P\), \(P\mapsto\Sh_P\), and \(P\mapsto\AllSh_P\) are continuous in their respective topologies. For each continuous test function, the measure convergence is uniform on compact coefficient families of fixed degree, and both set convergences are uniform in Hausdorff distance. Moreover,
\begin{equation}\label{eq:tail-core-intro}
 \Sh_P=\bigcap_{M\ge1}
       \overline{\bigcup_{n\ge M}\mathcal Z_{P,n}}.
\end{equation}
\end{theorem}

\subsection{Some properties of the limiting core and of the transient set}
The set
\begin{equation}\label{eq:trans}
 T_P:=\AllSh_P\setminus\Sh_P
\end{equation}
is called \emph{transient} and its points are called \emph{transient zeros}.

For a set \(E\subseteq\C\), let \(\acc (E)\) be the set of its accumulation points, i.e., limits of sequences of pairwise distinct points of \(E\). A set is \emph{perfect} if it is closed and has no isolated points. The \emph{perfect kernel} of a closed set is its largest perfect subset.

\begin{theorem}\label{thm:p1-perfect-summary} For any polynomial \(P\) of degree at least \(2\),

\smallskip
\noindent
{\rm (i)}
The limiting core \(\Sh_P\) is a compact connected subset of \(\Hull_P\) containing all roots and all critical points of \(P\), as well as the centroid of its root multiset.

\smallskip
\noindent
{\rm (ii)}
Every point of \(T_P\) is an actual finite-level zero, is isolated in \(\AllSh_P\), and occurs at only finitely many power levels. The set \(T_P\) is locally finite in \(\C\setminus\Sh_P\): every compact subset of \(\C\setminus\Sh_P\) meets it in finitely many points. Consequently the limiting core and the full shadow of \(P\) have the same interior and planar area.

\smallskip
\noindent
{\rm (iii)}
If \(P\) has at least two distinct roots, then
\[
 \acc(\AllSh_P)=\Sh_P=\acc(\Sh_P),
\]
and \(\Sh_P\) is the perfect kernel of \(\AllSh_P\). For every nonconstant \(P\),
\[
 \Sh_P=\bigcap_{k\ge1}\AllSh_{P^k},
 \qquad \AllSh_{P^k}\longrightarrow\Sh_P
 \quad\text{in Hausdorff distance as }k\to\infty.
\]
\smallskip
\noindent
{\rm (iv)}
The family \(\{\AllSh_{P^k}\}\) is nested along divisibility,
\[
 k\mid j\quad\Longrightarrow\quad
 \AllSh_{P^j}\subseteq\AllSh_{P^k},
\]
but need not be nested in the natural order \(k=1,2,\ldots\); indeed, \(\AllSh_{P^2}\) and \(\AllSh_{P^3}\) can be incomparable.
\end{theorem}

\medskip
\begin{remark}
By Theorem~\ref{thm:p1-perfect-summary}(i) the set \(\Sh_P\) is a canonical connected refinement of \(\conv Z(P)\) from the point of view of the Gauss--Lucas theorem, i.e., it is a canonically defined connected subset of \(\conv Z(P)\) containing all the critical points of \(P\), while \(\AllSh_P\) retains the exact finite-level exceptions. Moreover, in many numerical experiments the critical points of \(P\) lie on the boundary of \(\Sh_P\). Therefore \(\Sh_P\) might play an important role in the geometry of polynomials. For classical background on the geometry of polynomial zeros and critical points, see \cite[Chapters~I--II]{Marden}.
\end{remark}

\medskip
Interestingly, the set \(T_P\) of transient zeros can be infinite already if \(P\) has degree four.

\begin{theorem}[Infinitely many transients]\label{thm:intro-infinite}
For \(0<a<4\), let
\[
 \Pi_a(z)=z^4+az^3-(a-a^2/8)z^2+az+1.
\]
Its core is disjoint from \((0,\infty)\). For every sufficiently large odd \(n\), the polynomial \((\Pi_a^n)^{(2n)}\) has a simple positive zero
\[
 x_n=\frac{\sigma_a}{n}+O(n^{-2}),
\]
where \(\sigma_a\) is the unique solution in \((0,\pi/\sqrt{16-a^2})\) of
\[
 e^{(a+4)\sigma_a/2}
   =\frac4{\sqrt{4-a}}
       \cos\!\left(\frac{\sqrt{16-a^2}}2\sigma_a\right).
\]
Hence \(T_{\Pi_a}\) is infinite. Every sufficiently large even power level also contains a positive transient, after a bounded shift of the derivative order from \(2n\).
\end{theorem}

In particular, transient zeros may occur at arbitrarily large power levels, so the exact full shadow cannot in general be replaced by its asymptotic core after discarding any fixed finite truncation.

\medskip
Figure~\ref{fig:pi2-transients} shows the case \(a=2\). The core of \(\Pi_2\) consists of a segment of the negative real axis, two arms ending at the nonreal roots, and a two-dimensional region between them, which has a corner at the origin. The zeros of the first power levels spill over into the sector around the positive real axis, which is disjoint from the core. Every compact subset of that sector meets the zero sets of only finitely many power levels, but the positive transients of Theorem~\ref{thm:intro-infinite} approach the corner point \(0\) like \(\sigma_2/n\). The boundary arcs through the corner are not traced by endpoints of the supports \(\supp\mu_{\alpha,\Pi_2}\); see Subsection~\ref{rem:boundary-mechanisms}.

\begin{figure}[htbp]
\centering
\begin{minipage}[c]{0.44\textwidth}\centering
 \includegraphics[width=\linewidth]{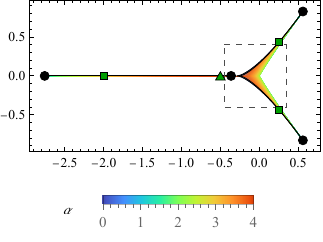}\\[1pt]
 {\small (a)}\\[8pt]
 \includegraphics[width=\linewidth]{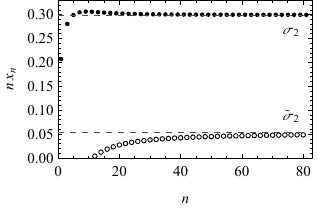}\\[1pt]
 {\small (c)}
\end{minipage}\hfill
\begin{minipage}[c]{0.54\textwidth}\centering
 \includegraphics[width=\linewidth]{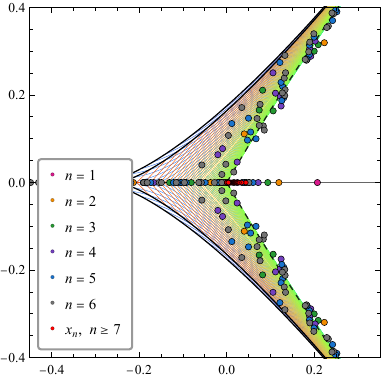}\\[1pt]
 {\small (b)}
\end{minipage}
\caption{The limiting core and the transient zeros of \(\Pi_2(z)=z^4+2z^3-\frac32z^2+2z+1\). (a) The core \(\Sh_{\Pi_2}\) as the union of the supports \(\supp\mu_{\alpha,\Pi_2}\), \(\alpha=k/24\), colored by \(\alpha\). The black curves are the paths of the endpoints of these supports as \(\alpha\) varies; the endpoints are critical values of \(F_{\Pi_2,\alpha}\). Black disks: roots; green squares: critical points; green triangle: the centroid \(-\frac12\). (b) The dashed box in (a), with the zeros of all descendants \((\Pi_2^n)^{(m)}\), \(0<m<4n\), at the power levels \(n\le6\) (colored by \(n\)), and the positive transients \(x_n\) of Theorem~\ref{thm:intro-infinite} for \(7\le n\le81\) (red; derivative order \(2n\) for odd \(n\) and \(2n+2\) for even \(n\)). All positive zeros are transients; among those of the levels \(n=2\) and \(n=3\) are the points \(\xi_2\approx0.1191\) and \(\xi_3\approx0.0936\) of Proposition~\ref{prop:non-nesting}. Dashed: \(\supp\mu_{2,\Pi_2}\). Near the corner at the origin the boundary of the core is formed by supports with \(\alpha\) close to \(2\). (c) \(nx_n\) against \(n\) for odd \(n\) (filled) and even \(n\) (open). The dashed lines are the limits \(\sigma_2\approx0.2995\) and \(\widetilde\sigma_2\approx0.0544\) of Theorem~\ref{thm:infinite-transients}.}
\label{fig:pi2-transients}
\end{figure}

\subsection{Short literature survey on repeated differentiation}

Repeated differentiation has been studied as a macroscopic flow of zero distributions through nonlocal transport equations and free-convolution models, including the real and rotationally invariant settings; see \cite{SteinerbergerTransport,ORourkeSteinerberger,SteinerbergerFree, HoskinsKabluchko,KabluchkoUnitary,COR}. Most of these works concern asymptotic zero distributions. In the periodic real-root setting, Kiselev and Tan \cite{KiselevTan} rigorously relate Steinerberger's nonlocal PDE to the evolution under repeated differentiation of a class of trigonometric polynomials and obtain an all-time quantitative bound on the error between each consecutive-root gap and the gap predicted by the corresponding PDE density. Mart\'inez-Finkelshtein and Rakhmanov develop a general Cauchy-transform flow and its Burgers/free-convolution connections in \cite{MFR}; their companion chapter \cite{MFRWeighted} gives a direct weighted-equilibrium treatment for zeros on \([-1,1]\) with initial arcsine distribution.

For a compactly supported finite measure \(\nu\), its Cauchy transform is \(C_\nu(z):=\int(z-\xi)^{-1}\dd\nu(\xi)\) off its support. Mart\'inez-Finkelshtein and Rakhmanov \cite[Theorem~1.1]{MFR} prove convergence of such transforms outside a common compact convex enclosure for general polynomial arrays. As they note after that theorem, an exterior transform need not have a unique representing planar measure. Their result therefore does not by itself identify the full planar limit considered here or exclude individual zeros inside the enclosure.

B\o gvad and the present authors established the fixed-ratio measure limit when \(P\) and \(P'\) have simple roots, see \cite[Theorem~1.1]{BHS}. Their Remark~1.3 identifies that restriction as apparently redundant, and Remark~1.8 already allows arbitrary sequences \(m_n/n\to\alpha\). Thus the sequential formulation is not itself new. The contributions here are removal of the simplicity restriction, convergence of the complete zero sets with coefficient stability, the connected structure of the limiting core, and the decomposition of the full shadow into the limiting core and transient zeros. At the lower endpoint, Theorem~\ref{thm:sublinear-counts} also gives exact cluster counts, uniformly for every sublinear range of derivative orders: \(nr_j-m\) zeros remain near a root \(a_j\), while \(m\varkappa_i\) zeros cluster near a nonroot critical point \(c_i\). The word \emph{shadow} also appears in \cite[Section~7, item~3]{BHS} for the region associated with limiting supports. The full shadow discussed here retains, in addition, all finite-level zeros; it is distinguished throughout from the limiting core.

\subsection{Structure of the paper}


Section~\ref{sec:barrier} establishes the contour geometry, and Section~\ref{sec:limits} proves the fixed-ratio limits and the behavior at roots needed for complete zero convergence. Section~\ref{sec:endpoints} treats the endpoint ratios. Section~\ref{sec:all-orders} combines these results to determine the limiting core and the full shadow. Section~\ref{sec:infinite-transients} describes the quartic example. At the ends of the relevant arguments we state explicitly how the later numbered results assemble into Theorems~\ref{thm:intro-fixed-ratio}--\ref{thm:intro-infinite}. Appendix~\ref{app:variation} gives the full first-variation formula at possibly degenerate saddles; Appendix~\ref{app:equal-height} records an example showing that equal critical heights do not imply equal contour contributions. Appendix~\ref{app:differentiator-spectrum} realizes the full shadow, limiting core, and transient set as the spectrum, Fredholm essential spectrum, and discrete spectrum of a bounded operator assembled from differentiators, and identifies its limiting mixed traces.

The figures were computed numerically. Zeros of descendants were found by high-precision polynomial root finding. The supports \(\supp\mu_{\alpha,P}\) were traced as curves along which two connecting saddles of the landscape defined in Section~\ref{sec:barrier} have the same height, starting from critical values of \(F_{P,\alpha}\). The connecting saddles were identified by following steepest-ascent lines of the landscape.

\FloatBarrier
\section{The logarithmic landscape, its barrier and contours}\label{sec:barrier}

\subsection{Basic notation and definitions}
Write
\begin{equation}\label{eq:root-data}
 P(z)=\prod_{j=1}^{s}(z-a_j)^{r_j},\qquad
 r_j\in\N,\qquad \sum_{j=1}^{s}r_j=d,
\end{equation}
where the \(a_j\) are distinct, and put \(p_j=r_j/d\). The root multiset repeats each \(a_j\) exactly \(r_j\) times. Its root probability measure and centroid are
\[
 \mu_P^{\mathrm{root}}=\frac1d\divz P=\sum_{j=1}^{s}p_j\delta_{a_j},
 \qquad
 \centroid_P:=\sum_{j=1}^{s}p_j a_j.
\]
The logarithmic derivative of \(P\) and the Cauchy transform of its root measure are denoted by
\[
 \LogDer_P:=\frac{P'}P,\qquad
 \rootC_P:=C_{\mu_P^{\mathrm{root}}}=\frac1d\LogDer_P.
\]
Thus \(\Hull_P\) denotes the convex hull of \(Z(P)\), \(\supp\mu_{\alpha,P}\) the support of \(\mu_{\alpha,P}\), and \(\Sh_P\) the limiting core of \(P\).

\medskip
For a compactly supported finite Borel measure \(\nu\), our conventions are
\begin{equation}\label{eq:potential-conventions}
 U_\nu(z)=\int\log|z-\xi|\dd\nu(\xi),\qquad
 C_\nu(z)=\int\frac{\dd\nu(\xi)}{z-\xi}.
\end{equation}
The Cauchy transform is locally integrable in the plane and holomorphic off \(\supp\nu\). With \(z=x+iy\) and
\[
 \partial_z:=\tfrac12(\partial_x-i\partial_y),\qquad
 \partial_{\bar z}:=\tfrac12(\partial_x+i\partial_y),\qquad
 \Delta:=\partial_x^2+\partial_y^2,
\]
we have \(2\partial_zU_\nu=C_\nu\) off \(\supp\nu\) and \(\Delta U_\nu=2\pi\nu\), or equivalently \(\partial_{\bar z}C_\nu=\pi\nu\), as distributions; see \cite{Ransford} for these potential-theoretic conventions.

An \emph{atom} of \(\nu\) is a point of positive \(\nu\)-mass, and \(\nu^{\mathrm{at}}\) denotes the sum of its point masses. For a Borel map \(f\), the pushforward is \((f_*\nu)(E):=\nu(f^{-1}(E))\). We identify \(\C\) with \(\R^2\) and write \(\Area\) for planar area.

The notation \(O(g_n)\) and \(o(g_n)\) has its usual uniform meaning in the stated parameter range; subscripts indicate allowed dependence of the constants. Primes on one-variable holomorphic functions denote holomorphic derivatives; accumulation set is denoted by \(\acc\). The coefficient of \(t^k\) in a Taylor expansion of \(f(t)\) is \([t^k]f(t)\).

\subsection{The landscape and the reduced saddle equation}
Fix \(P\), \(z\notin Z(P)\) and \(0<\alpha<d\). Below all real logarithms are natural; complex logarithms denote local branches. Define the \emph{logarithmic landscape} of this data as
\begin{equation}\label{eq:landscape}
 L_{P,z,\alpha}(w):=\log|P(w)|-\alpha\log|w-z|.
\end{equation}
Locally this is the real part of the phase function
\[
 \Phi_{P,z,\alpha}(w):=\log P(w)-\alpha\log(w-z).
\]
The landscape tends to \(+\infty\) at \(w=z\) and \(w=\infty\), its two \emph{positive ends}, and to \(-\infty\) at the roots, its \emph{negative ends}. A \emph{finite saddle} is a point \(w\notin Z(P)\cup\{z\}\) at which \(\Phi_{P,z,\alpha}'(w)=0\); equivalently,
\begin{equation}\label{eq:finite-saddle}
 (w-z)P'(w)=\alpha P(w).
\end{equation}
The value of \(L_{P,z,\alpha}\) at a finite saddle is its \emph{critical height}.

To remove the repeated-root factors from the polynomial equation, set
\begin{equation}\label{eq:reduced-polynomials}
 Q_P(w):=\prod_{j=1}^s(w-a_j),\qquad
 N_P(w):=\sum_{j=1}^s r_j\prod_{k\ne j}(w-a_k).
\end{equation}
Then
\[
 \LogDer_P(w)=\frac{P'(w)}{P(w)}=\frac{N_P(w)}{Q_P(w)}.
\]
The polynomial \(N_P\) has degree \(s-1\) and leading coefficient \(d\), and it is coprime to \(Q_P\). For a finite saddle, \eqref{eq:finite-saddle} is therefore equivalent to
\begin{equation}\label{eq:reduced-saddle}
 (w-z)N_P(w)-\alpha Q_P(w)=0.
\end{equation}
This reduced saddle polynomial has degree \(s\) and leading coefficient \(d-\alpha\ne0\). The multiplicity of a finite saddle is its multiplicity as a zero of this reduced polynomial. It is \emph{simple}, or \emph{nondegenerate}, when that multiplicity is one; since \(Q_P(w)(w-z)\ne0\) at a finite saddle, this is equivalent to \(\Phi_{P,z,\alpha}''(w)\ne0\).

Define the reduced quotient and the rational saddle map by
\begin{equation}\label{eq:F}
 h_P(w):=\frac{Q_P(w)}{N_P(w)}=\frac{P(w)}{P'(w)},\qquad
 F_{P,\alpha}(w):=w-\alpha h_P(w).
\end{equation}
The saddle equation is \(F_{P,\alpha}(w)=z\); no zero of \(Q_P\) is a finite saddle when \(z\notin Z(P)\). At \(\alpha=1\), the rational map is exactly Newton's map; for general \(\alpha\), its Newton correction is multiplied by \(\alpha\).

\subsection{Separation of the positive ends}
For \(\ell\in\R\), consider the compact subset of the Riemann sphere \(\bar \C\) given by the condition
\[
 E_\ell(P,z,\alpha):=
 \{w\in\C\setminus(Z(P)\cup\{z\}):L_{P,z,\alpha}(w)\ge \ell\}
 \cup\{z,\infty\}.
\]
The \emph{barrier level} is
\begin{equation}\label{eq:barrier-def}
 \Blev_P(z,\alpha):=
 \sup\{\ell\in\R:z\text{ and }\infty\text{ lie in the same connected component of
 }E_\ell(P,z,\alpha)\}.
\end{equation}
It is the largest level at which the two positive ends remain connected inside \(E_\ell\). A value \(\ell\) is \emph{regular} if the level set \(\{L_{P,z,\alpha}=\ell\}\) contains no finite saddle.

Let \(\mathscr J_z\) be the family of piecewise \(C^1\) Jordan curves in \(\C\setminus(Z(P)\cup\{z\})\) with winding number one about \(z\). The next lemma proves the min--max description of the barrier together with the separation properties needed below.

\medskip
Suppress the parameters and write \(L=L_{P,z,\alpha}\) and \(E_\ell= E_\ell(P,z,\alpha)\). Near a root \(a_j\) of \(P\), \(L(w)=r_j\log|w-a_j|+O(1)\); near \(z\), \(L(w)=-\alpha\log|w-z|+O(1)\); and at infinity, \(L(w)=(d-\alpha)\log|w|+o(1)\).

The local normal form at a regular point is a smooth level arc, and at a finite saddle it is a finite union of analytic arcs with alternating sectors. Consequently the level and superlevel sets used below are locally path-connected; their connected components are path-connected.

\begin{lemma}[Separation and the min--max formula]\label{lem:barrier}
For any real \(\ell\), the set \(E_\ell\) has at most two connected components. For sufficiently large positive \(\ell\) it has two, one containing \(z\) and one containing infinity respectively. For sufficiently large negative \(\ell\) the set \(E_\ell\) is connected. If \(B=\Blev_P(z,\alpha)\), then \(E_\ell\) is connected for \(\ell\le B\) and has two components for \(\ell>B\). Moreover,
\[
 B=\inf_{\gamma\in\mathscr J_z}\max_{w\in\gamma}L(w).
\]
The value \(B\) is a critical height attained at a finite saddle.
\end{lemma}
\begin{proof}
Every superlevel component contains a positive end. Otherwise the component is compactly contained in the harmonicity domain of \(L\) and has nonempty interior: a regular boundary point has a superlevel side, whereas an isolated stationary point cannot be a local maximum of a nonconstant harmonic function by the maximum principle. There are only two positive ends, so there are at most two components.

For large positive \(\ell\), the superlevel set consists of disjoint neighborhoods of \(z\) and infinity. For sufficiently negative \(\ell\), it contains the complement of small disjoint disks about the roots; an additional component would again have to contain a positive end. It is therefore connected. As the level decreases, the positive-end components can merge but cannot separate. The connected levels form a half-line whose endpoint belongs to it, by the nested-intersection property of compact connected sets. The high- and low-level descriptions show that this endpoint is finite.

For \(\ell<B\), there is a path in \(E_\ell\) from \(z\) to infinity, after an arbitrarily small decrease in the level if needed. Every curve in \(\mathscr J_z\) meets this path, and hence has maximum at least \(\ell\). Conversely, if \(\ell>B\) is regular, a Jordan curve in \(\{L<\ell\}\) separates the two compact superlevel components. One may take the outer boundary of a small regular neighborhood of the component containing \(z\) and perturb it to avoid the finitely many roots. Letting \(\ell\downarrow B\) proves the min--max formula.

Finally, on a compact level band without saddles, the vector field \(\nabla L/|\nabla L|^2\) identifies neighboring levels. Finite level bands avoid all ends. Their topology can therefore change only at a saddle, so \(B\) is a critical height.
\end{proof}

\begin{lemma}[Continuity of the barrier]\label{lem:barrier-cont}
For monic polynomials \(P\) of a fixed degree \(d\), \(\Blev_P(z,\alpha)\) is jointly continuous in the coefficients of \(P\), in \(z\notin Z(P)\), and in \(0<\alpha<d\).
\end{lemma}
\begin{proof}
Write \(B=\Blev_P(z,\alpha)\). For the upper estimate, choose a curve in \(\mathscr J_z\) whose maximum is within \(\eps\) of \(B\). The curve avoids the roots, the landscape varies uniformly on it, and its winding number about the moving observation point is unchanged by a small perturbation.

For the lower estimate, join the positive ends by a path in \(\{L>B-\eps\}\) and truncate it in small neighborhoods of \(z\) and infinity. The remaining compact path avoids all roots and persists under perturbation. The end neighborhoods can be chosen so that the perturbed landscapes are uniformly above \(B-2\eps\): the polynomial stays nonzero near the moving observation point, and \(\alpha\) and \(d-\alpha\) stay positive. Reattach the path to the two ends and apply \eqref{eq:barrier-def}. Since \(\eps\) is arbitrary, the two estimates prove continuity. Root collisions away from \(z\) do not affect the argument.
\end{proof}

\subsection{Connecting saddles and barrier contours}
Fix \((P,z,\alpha)\) and write \(B=\Blev_P(z,\alpha)\). The strict superlevel set \(\{L_{P,z,\alpha}>B\}\), with its positive ends adjoined, has exactly two components, denoted by \(\Omega_z\) and \(\Omega_\infty\). Each contains a positive end, by the maximum principle; a path joining the ends in the strict superlevel would contradict the definition of \(B\). Moreover,
\[
 E_B=\overline{\Omega_z}\cup\overline{\Omega_\infty}.
\]
Indeed, every point of the critical level has a higher-level sector. The closures meet because \(E_B\) is connected. They meet only at saddles, since the higher side of a regular level arc is locally connected.

A \emph{connecting saddle} is an element of \(\overline{\Omega_z}\cap\overline{\Omega_\infty}\). A saddle is \emph{at barrier height} if its height equals \(B\); such a saddle need not be connecting. If there is exactly one connecting saddle and it is simple, we call it the \emph{separating saddle}.

A \emph{barrier contour} is a positively oriented piecewise smooth Jordan curve in \(\mathscr J_z\) which passes once through each connecting saddle and has \(L_{P,z,\alpha}<B\) everywhere else. A \emph{one-saddle separating contour} is a barrier contour with one simple connecting saddle, smooth through that saddle, along which the landscape has a nondegenerate maximum: its second derivative in a regular real parametrization is negative. Thus \emph{connecting} is a topological condition, \emph{separating} means unique and simple, and a one-saddle separating contour is the analytic contour used in the Gaussian asymptotic.

Figure~\ref{fig:landscape} shows these objects for a quartic polynomial. At an observation point off the limiting support \(\supp\mu_{\alpha,P}\) the components \(\Omega_z\) and \(\Omega_\infty\) touch at a single simple saddle, see Lemmas~\ref{lem:multiple-gates} and~\ref{lem:harmonic-gate} below. At a point of \(\supp\mu_{\alpha,P}\) whose connecting saddles are simple, there are at least two of them, since a single simple connecting saddle makes the barrier harmonic nearby.

\begin{figure}[htbp]
\centering
\includegraphics[width=0.49\textwidth]{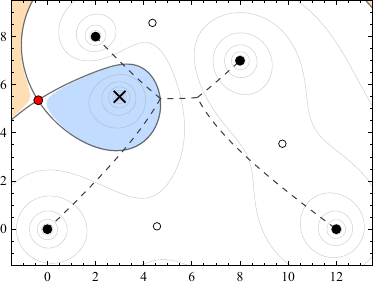}\hfill
\includegraphics[width=0.49\textwidth]{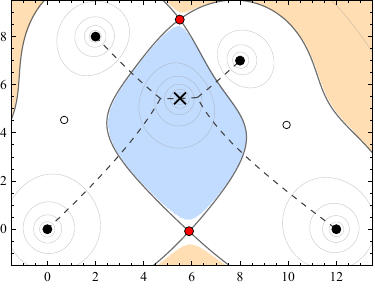}
\caption{The landscape \(L_{P,z,\alpha}\) in the \(w\)-plane for \(P(z)=z(z-12)(z-2-8i)(z-8-7i)\) and \(\alpha=\frac{13}{10}\): contour lines (gray) and the critical level \(\{L_{P,z,\alpha}=\Blev_P(z,\alpha)\}\) (black). The two components \(\Omega_z\) (blue) and \(\Omega_\infty\) (orange) of \(\{L_{P,z,\alpha}>\Blev_P(z,\alpha)\}\) touch at the connecting saddles (red); the other finite saddles are white, the roots of \(P\) are black disks, and the cross is the observation point \(z\). Dashed: \(\supp\mu_{\alpha,P}\), drawn in the same plane for orientation. Left: \(z=3+\frac{11}2i\notin\supp\mu_{\alpha,P}\), with one simple connecting saddle. Right: \(z\in\supp\mu_{\alpha,P}\), with two connecting saddles of the same height.}
\label{fig:landscape}
\end{figure}

\begin{remark}[Which saddles matter]\label{rem:not-max}
A critical height can exceed \(B\): the barrier is not the maximum over all saddles. Nor does equality with \(B\) force a saddle to contribute to a contour. What matters is contact between the two positive-end components. The distinction already occurs for quadratic polynomials.
\end{remark}

Near a finite saddle \(w_j\), choose logarithms and write
\[
 \Phi(w)=\log P(w)-\alpha\log(w-z),\qquad \Xi_j=\Phi(w_j).
\]
Its critical height is \(\Ree\Xi_j\). Its \emph{phase order} \(q_j\ge2\) is the first nonzero derivative order of \(\Phi-\Xi_j\):
\[
 \Phi'(w_j)=\cdots=\Phi^{(q_j-1)}(w_j)=0,
 \qquad \Phi^{(q_j)}(w_j)\ne0.
\]
Factoring this zero and taking a local holomorphic \(q_j\)th root gives a coordinate \(v_j\), with \(v_j(w_j)=0\) and \(v_j'(w_j)\ne0\), such that
\[
 \Phi(w)=\Xi_j-v_j(w)^{q_j}.
\]
The \emph{higher sectors} and \emph{lower sectors} are the alternating regions where \(-\Ree(v_j^{q_j})\) is positive and negative, respectively. The central rays of the lower sectors, called \emph{decreasing rays}, have the form \(v_j=\zeta t\) with \(t\ge0\) and \(\zeta^{q_j}=1\). When \(q_j=2\), the two decreasing rays form the \emph{descending axis}. A simple saddle has phase order two. Enumerating the connecting saddles as \(w_1,\ldots,w_k\), the reduced saddle equation gives
\begin{equation}\label{eq:phase-order-budget}
 q_j\ge2,\qquad \sum_{j=1}^k(q_j-1)\le s.
\end{equation}

We require a precise fact about the outer boundary of a level set. In the next lemma, a \emph{finite level graph} means a compact level set made of finitely many smooth arcs meeting at finitely many saddle vertices; a smooth closed component is regarded as a cycle with an auxiliary vertex. To \emph{fill} a compact planar set is to adjoin every bounded component of its complement.

\begin{lemma}[Boundary cycle of a full finite-graph continuum]
\label{lem:full-boundary-cycle}
Let \(K\subset\C\) be nonempty, compact, connected, and \emph{regular closed}, that is, \(K=\overline{\operatorname{int}K}\). Assume that \(U:=\C\setminus K\) is connected, and suppose that \(\partial K\) is contained in a finite embedded graph \(G\). At every vertex \(v\in\partial K\), assume that the local sectors belonging to \(U\) form one nonempty proper cyclic block. Then \(\partial K\) is a single Jordan cycle in \(G\). At each vertex it uses exactly the two half-edges that bound the \(U\)-block.
\end{lemma}
\begin{proof}
In a small neighborhood of a regular point of \(G\), each of the two sides of the edge lies entirely in \(U\) or entirely in \(\operatorname{int}K\), because \(\partial K\subset G\). If both sides lie in \(U\), regular closedness implies that the intervening edge also lies in \(U\): a point of \(K\) there could not be approached by interior points of \(K\). If both sides lie in \(\operatorname{int}K\), closedness of \(K\) puts the whole neighborhood in \(K\). Thus a regular point lies in \(\partial K\) exactly when the two sides have opposite labels, and then the boundary is locally the entire edge. In particular, each open edge of \(G\) is either wholly in \(\partial K\) or disjoint from it.

The same argument applies to the half-edges in a vertex neighborhood: a half-edge belongs to \(\partial K\) exactly when its adjacent sectors have different labels. Half-edges between two \(U\)-sectors belong to \(U\), so no boundary slit is overlooked. The single nonempty proper \(U\)-block therefore has exactly two boundary half-edges. Every boundary vertex has degree two in the boundary subgraph. This is a nonempty finite embedded graph with degree two at every vertex.

The plane is unicoherent: two closed connected sets covering it have connected intersection. Applying this to \(K\) and \(\overline U\) shows that \(\partial K=K\cap\overline U\) is connected. A connected finite graph of degree two is a single cycle, and its embedded realization is a Jordan curve. This includes a loop edge or a pair of parallel edges.
\end{proof}

\begin{lemma}[Outer-boundary regularization]\label{lem:outer-regularization}
Let \(G=\{L_{P,z,\alpha}=B\}\), let \(\Omega_{\mathrm{ext}}\) be the unbounded component of \(\C\setminus\overline{\Omega_z}\), and set
\[
 \widehat{\Omega}_z:=\C\setminus\Omega_{\mathrm{ext}}.
\]
Thus \(\widehat{\Omega}_z\) is the filled set of \(\overline{\Omega_z}\). Then \(\partial\widehat{\Omega}_z\) is one Jordan cycle in the finite level graph \(G\). Every connecting saddle lies on this cycle and is visited once. At such a saddle the higher sectors belonging to \(\Omega_z\) and \(\Omega_\infty\) form two cyclic blocks. The two lower sectors between these blocks are the exterior passages on the two sides of the cycle.

Choose disjoint small disks about the vertices of this cycle and disjoint thin collars about its remaining edges. Its edges can be moved into \(\{L_{P,z,\alpha}<B\}\) within the collars. In each nonconnecting vertex disk they can be joined in the lower sector on the exterior side. In a connecting vertex disk they can instead be joined through the vertex along any prescribed pair of central rays in the two exterior lower sectors. These replacements give a single Jordan curve, with no repeated vertex or additional component.
\end{lemma}
\begin{proof}
The level set \(G\) stays in a compact region away from all ends. The reduced saddle equation gives only finitely many critical points. At a regular point the real-analytic implicit-function theorem gives one smooth level arc, while at a saddle the phase normal form gives finitely many analytic branches. If \(G\) had infinitely many distinct edges or closed regular components, points on them would accumulate in \(G\); the corresponding local normal form would then contain more than its one regular branch or its finitely many saddle branches. Thus \(G\) is a finite embedded graph.

Put \(K=\widehat{\Omega}_z\) and \(U=\Omega_{\mathrm{ext}}\). Then \(K\) is compact and connected, \(U=\C\setminus K\) is connected, and \(\partial K\subseteq\partial\Omega_z\subseteq G\). Moreover \(\Omega_\infty\subset U\), so filling cannot capture the infinity component. The set \(\overline{\Omega_z}\) is regular closed. Each bounded component of its complement is open and is included in \(K\); together with \(\Omega_z\) these components have closure \(K\). Thus \(K\) is also regular closed, as required in Lemma~\ref{lem:full-boundary-cycle}.

We use the following noncrossing observation. Two disjoint connected open sets cannot occupy four alternating sectors at one vertex. Indeed, join two sector points in each open set, truncate the two disjoint arcs outside a small disk, and extend their endpoints radially to the circle. In the complementary disk on the sphere, two disjoint arcs would then join alternating boundary points, contrary to the Jordan curve theorem.

At a connecting saddle, apply this observation to \(\Omega_z\) and \(\Omega_\infty\), reading only the cyclic order of the higher sectors. Since every higher sector has one of these two labels and both labels occur, the sectors carrying each label form one cyclic block.

Now consider all local sectors at a vertex of \(\partial K\). Every sector outside \(\overline{\Omega_z}\) belongs to one component of \(\C\setminus\overline{\Omega_z}\). Two adjacent such sectors must have the same component label: otherwise a thin neighborhood crossing the intervening level ray, but avoiding the vertex, would join them in \(\C\setminus\overline{\Omega_z}\). Hence sectors with different complement labels are separated in the cyclic order by an \(\Omega_z\)-sector. Consequently, if the unbounded component \(U\) occurred in two cyclic blocks, each of the two intervals between those blocks would contain an \(\Omega_z\)-sector. This would give four alternating sectors belonging to the two connected open sets \(U\) and \(\Omega_z\), contrary to the noncrossing observation. Thus the sectors belonging to \(U\) form one block at every vertex of \(\partial K\).

Lemma~\ref{lem:full-boundary-cycle} now shows that \(\partial K\) is one Jordan cycle and that each of its vertices is visited once. At a connecting saddle, the two interfaces between the \(\Omega_z\)- and \(\Omega_\infty\)-blocks are lower sectors in \(U\); these are precisely the two exterior passages. Every connecting saddle lies on the cycle, because it is approached both from \(\Omega_z\) and from \(\Omega_\infty\subset U\).

For the regularization, choose vertex disks with disjoint closures, each meeting the cycle only in its two incident arcs, and choose disjoint rectangular collars around the remaining edges. Move each regular edge slightly into its lower \(U\)-side. In a nonconnecting vertex disk, join the displaced ends inside the unique exterior lower sector. In a connecting disk, join each displaced end to the prescribed central ray in its exterior lower sector and join the two rays at the vertex. Each replacement is a simple arc with the same two boundary endpoints as the original cycle segment. The disks and collars are disjoint, so the replacements preserve one Jordan curve and introduce neither a repeated vertex nor an additional component. The curve keeps \(z\) on its bounded side and infinity on its unbounded side.
\end{proof}

\begin{lemma}[A barrier contour through all connecting saddles]
\label{lem:arbitrary-contour}
There is a barrier contour \(\gamma\) for \((P,z,\alpha)\). Near each \(w_j\) there is a holomorphic coordinate \(v_j\), with \(v_j(w_j)=0\) and \(v_j'(w_j)\ne0\), such that
\[
 \Phi(w)=\Xi_j-v_j(w)^{q_j}.
\]
The contour enters and leaves \(w_j\) along distinct decreasing rays \(v_j=\zeta_j^-t\) and \(v_j=\zeta_j^+t\), \(t\ge0\), where \((\zeta_j^\pm)^{q_j}=1\). Outside small coordinate neighborhoods its compact remainder has a strict height gap below \(B=\Blev_P(z,\alpha)\).
\end{lemma}
\begin{proof}
Use the local coordinates constructed above and the two lower sectors between the higher-sector blocks identified in Lemma~\ref{lem:outer-regularization}. Apply that lemma with their central rays prescribed in each connecting disk. The resulting curve winds once about \(z\), avoids all roots, and has \(L<B\) except at the connecting saddles. On the compact remainder outside smaller disks, the continuous function \(B-L\) is strictly positive, hence has a positive minimum. This is the claimed gap. When \(k=1\) and \(q_1=2\), the two rays form a smooth descending axis and the contour is a one-saddle separating contour.
\end{proof}

Figure~\ref{fig:connecting-contour} illustrates the local passages for a simple saddle and for a phase-order-four saddle; the global curve is the regularized outer boundary from Lemma~\ref{lem:outer-regularization}.

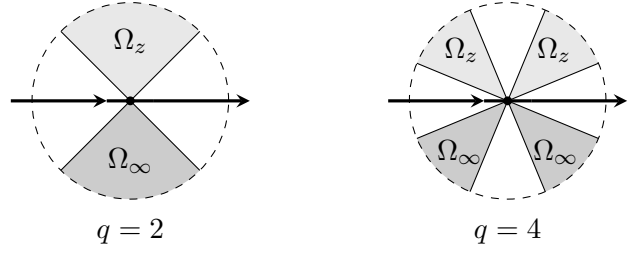
\begin{figure}[ht]
\centering
\begin{tikzpicture}[scale=1.04,>=stealth]
 \begin{scope}[xshift=-2.4cm]
  \fill[black!9] (0,0)--(45:1.25) arc (45:135:1.25)--cycle;
  \fill[black!20] (0,0)--(225:1.25) arc (225:315:1.25)--cycle;
  \draw[dashed] (0,0) circle (1.25);
  \foreach \a in {45,135,225,315}{\draw[thin] (0,0)--(\a:1.25);}
  \draw[very thick,->] (-1.52,0)--(-.3,0);
  \draw[very thick] (-.3,0)--(.3,0);
  \draw[very thick,->] (.3,0)--(1.52,0);
  \fill (0,0) circle (1.6pt);
  \node at (0,.76) {\(\Omega_z\)}; \node at (0,-.76) {\(\Omega_\infty\)};
  \node[below] at (0,-1.38) {\(q=2\)};
 \end{scope}
 \begin{scope}[xshift=2.4cm]
  \foreach \a/\b in {22.5/67.5,112.5/157.5}
    {\fill[black!9] (0,0)--(\a:1.25) arc (\a:\b:1.25)--cycle;}
  \foreach \a/\b in {202.5/247.5,292.5/337.5}
    {\fill[black!20] (0,0)--(\a:1.25) arc (\a:\b:1.25)--cycle;}
  \draw[dashed] (0,0) circle (1.25);
  \foreach \a in {22.5,67.5,112.5,157.5,202.5,247.5,292.5,337.5}
    {\draw[thin] (0,0)--(\a:1.25);}
  \draw[very thick,->] (-1.52,0)--(-.3,0);
  \draw[very thick] (-.3,0)--(.3,0);
  \draw[very thick,->] (.3,0)--(1.52,0);
  \fill (0,0) circle (1.6pt);
  \node at (.59,.65) {\(\Omega_z\)}; \node at (-.59,.65) {\(\Omega_z\)};
  \node at (.61,-.65) {\(\Omega_\infty\)}; \node at (-.61,-.65) {\(\Omega_\infty\)};
  \node[below] at (0,-1.38) {\(q=4\)};
 \end{scope}
\end{tikzpicture}
\caption{Local contour passages in the coordinate \(L-B=-\Ree(v^q)\). Shaded sectors are above the barrier; unshaded sectors are below it. Each superlevel label occupies one cyclic block. The thick arrows use the two lower sectors separating those blocks. For \(q=4\) other lower sectors lie within a block and are not the outer-boundary passage. The global contour is obtained by regularizing the single outer boundary cycle.}
\label{fig:connecting-contour}
\end{figure}
\FloatBarrier

\subsection{Simple connecting saddles and persistent contours}
At a simple saddle, a \emph{Morse coordinate} is a holomorphic local coordinate \(\zeta\) in which the landscape is \(\psi+\Ree(\zeta^2)\), where \(\psi\) is the critical height. A small coordinate disk will be called a Morse disk. In this coordinate the descending axis is the imaginary axis, with height \(\psi-|\zeta|^2\).

\begin{lemma}[The maximum over connecting saddles]\label{lem:multiple-gates}
Suppose the connecting saddles at \((P,z_0,\alpha_0)\) are simple, and enumerate them as \(w_1,\ldots,w_k\). Other saddles may be multiple or have coincident heights. The connecting saddles continue simply on a parameter neighborhood, with real critical heights \(\psi_j(z,\alpha)\), and
\begin{equation}\label{eq:multiple-gate-max}
 \Blev_P(z,\alpha)=\max_{1\le j\le k}\psi_j(z,\alpha).
\end{equation}
This description persists under coefficient perturbation. If \(k=1\), the saddle has a locally persistent one-saddle separating contour.
\end{lemma}
\begin{proof}
Choose disjoint small Morse disks about the connecting saddles at the base parameter. At each saddle, the two higher-level sectors lead to different positive ends. Join them to the ends within the corresponding strict components. Outside the saddle disk, the compact part of each path has a positive height margin. These paths above the barrier and their end attachments persist under perturbation. The path through \(w_j\) therefore proves \(\Blev_P\ge\psi_j\) nearby.

For the reverse inequality use the contour of Lemma~\ref{lem:arbitrary-contour} at the base parameter. It passes through all connecting saddles once and has a strict height gap away from their disjoint Morse disks. Its single outer-boundary cycle is supplied by Lemma~\ref{lem:outer-regularization}.

After smaller saddle neighborhoods are removed, the compact remainder has a strict gap. Under perturbation, it and its connections to the local decreasing arcs remain below every nearby \(\psi_j\). The arcs pass through the continued simple saddles, so the maximum on the perturbed contour is \(\max_j\psi_j\). The min--max formula proves the upper bound in \eqref{eq:multiple-gate-max}. For \(k=1\) this is a one-saddle separating contour. All constructions persist in the coefficients as well as in \((z,\alpha)\).
\end{proof}

A \emph{regular chamber} is a connected open parameter region in which no saddle branches coalesce and no two continued critical heights coincide. For fixed \(P\) and \(\alpha\), we call an observation point \(z\) \emph{regular} if all finite saddles are simple and have pairwise distinct critical heights.

\begin{lemma}[Regular observation points]\label{lem:regular}
Fix \(P\) and \(0<\alpha<d\). Almost every observation point, with respect to planar Lebesgue measure, is regular. At a regular point there is exactly one saddle at barrier height, and it has a locally persistent one-saddle separating contour. Persistence holds also under coefficient perturbations.
\end{lemma}
\begin{proof}
The numerator and denominator of \(F_{P,\alpha}\) are coprime, since \(\alpha\ne0\) and \(Q_P,N_P\) are coprime. This rational map therefore has degree \(s\). Its branch values, meaning images of its critical points, form a finite set. Away from them and the roots of \(P\), write the inverse branches locally as \(\omega_1(z),\ldots,\omega_s(z)\). Choose local logarithms and put
\[
 \Xi_i(z,\alpha):=\Phi_{P,z,\alpha}(\omega_i(z,\alpha)).
\]
The saddle equation gives
\begin{equation}\label{eq:phase-z}
 \partial_z\Xi_i=\frac{\alpha}{\omega_i-z}.
\end{equation}
For \(i\ne j\) the derivative of \(\Xi_i-\Xi_j\) never vanishes. Equality of their real critical heights thus defines a smooth real-analytic arc locally and has planar measure zero. A countable cover proves the almost-everywhere assertion.

At a regular point, Lemma~\ref{lem:barrier} puts the barrier at a critical height. Distinctness of the heights leaves one connecting saddle, and regularity makes it simple. Lemma~\ref{lem:arbitrary-contour} provides the one-saddle separating contour; Lemma~\ref{lem:multiple-gates} provides its persistence. In the quadratic normal coordinate the landscape has the form \(B+\Ree(\zeta^2)\), and the contour follows the imaginary axis, where it equals \(B-|\zeta|^2\). This is the separating-contour geometry used in \cite[Section~4.2]{BHS}; the argument here only requires the signs at the ends, not simple roots of \(P\) or \(P'\).
\end{proof}

\subsection{Harmonicity forces a single simple saddle}
\begin{lemma}[A harmonic barrier has a one-saddle separating contour]\label{lem:harmonic-gate}
Fix \(P\) and \(0<\alpha<d\). If \(\Blev_P(\cdot,\alpha)\) is harmonic on an open disk \(D\subset\C\setminus Z(P)\), then every \(z\in D\) has a one-saddle separating contour whose unique maximum is one simple saddle. The contour and its strict gap away from the saddle persist under small changes of \(P\), \(z\), and \(\alpha\). No simplicity or distinct-height condition is imposed on the other saddles.
\end{lemma}
\begin{proof}
Fix \(z_0\in D\), put \(L_0=L_{P,z_0,\alpha}\) and \(B_0=\Blev_P(z_0,\alpha)\), and choose a connecting saddle \(w_0\). Its local level set has finitely many alternating higher and lower sectors. Join a higher sector from \(\Omega_{z_0}\) to one from \(\Omega_\infty\) through \(w_0\), and continue within the strict components to the ends. Truncation gives a compact path \(\gamma_0\) on which \(L_0\) has its unique minimum \(B_0\) at \(w_0\). The end attachments stay strictly above \(B_0\) under perturbation, so for \(z\) near \(z_0\),
\[
 \Blev_P(z,\alpha)\ge \underline B(z):=\min_{w\in\gamma_0}L_{P,z,\alpha}(w),
 \qquad \underline B(z_0)=B_0.
\]
The truncated path avoids \(z_0\) and all roots. Its unique minimizing point and uniform first-order parameter expansion give
\begin{equation}\label{eq:gate-gradient}
 \underline B(z_0+\delta)=B_0+
 \Ree\!\left(\frac{\alpha}{w_0-z_0}\delta\right)+o(|\delta|).
\end{equation}
More explicitly, minimizers for \(z_0+\delta\) approach \(w_0\) by uniqueness and compactness. Evaluating first at \(w_0\) and then at a perturbed minimizer gives the two bounds in \eqref{eq:gate-gradient}.

The barrier is differentiable at \(z_0\). Compare its differential with that of the lower touching function in the directions \(\delta\) and \(-\delta\). We obtain
\begin{equation}\label{eq:gate-selected}
 2\partial_z\Blev_P(z_0,\alpha)=\frac{\alpha}{w_0-z_0}.
\end{equation}
Distinct connecting saddles would give distinct right-hand sides. There is therefore exactly one connecting saddle at each point of \(D\), and the derivative in \eqref{eq:gate-selected} never vanishes. Harmonicity makes
\[
 w(z):=z+\frac{\alpha}{2\partial_z\Blev_P(z,\alpha)}
\]
holomorphic on \(D\). This is the connecting saddle, and \(F_{P,\alpha}(w(z))=z\). Differentiation gives \(F_{P,\alpha}'(w(z))w'(z)=1\), so the saddle is simple. Lemma~\ref{lem:multiple-gates}, with \(k=1\), now supplies the one-saddle separating contour and its persistence in every parameter.
\end{proof}

\section{Zero limits for fixed ratio \texorpdfstring{\(\alpha\in(0,d)\)}{alpha in (0,d)}; proof of Theorem~\ref{thm:intro-fixed-ratio}}
\label{sec:limits}

We start by determining the limiting root-counting measure at an interior ratio \(\alpha=\lim_{n\to \infty}\frac{m_n}{n}\in (0,d)\). A one-saddle separating contour identifies its potential. Harmonicity of that potential then excludes zeros off the support. The only additional issue is behavior at a root when its atomic mass disappears; we treat that threshold explicitly before proving Hausdorff convergence.

The leading-coefficient normalization will be expressed through
\[
 \mathfrak h_d(\alpha):=d\log d-\alpha\log\alpha-(d-\alpha)\log(d-\alpha),
 \qquad 0<\alpha<d,
\]
with continuous endpoint values \(\mathfrak h_d(0)=\mathfrak h_d(d)=0\).

\subsection{The one-saddle asymptotic and logarithmic potentials}
\begin{proposition}[Uniform one-saddle asymptotic]\label{prop:uniform-saddle}
Suppose \((z_0,\alpha_0)\) admits a one-saddle separating contour as in Lemma~\ref{lem:regular} or Lemma~\ref{lem:harmonic-gate}. On a neighborhood, let \(w_*(z,\alpha)\) be the separating saddle and put \(\Xi_*(z,\alpha)=\Phi_{P,z,\alpha}(w_*(z,\alpha))\). These functions can be chosen holomorphic in \(z\) and real-analytic in \(\alpha\). There is a nonvanishing amplitude \(\mathfrak a(z,\alpha)\) of the same regularity such that, at the exact ratio \(\alpha=m/n\),
\begin{equation}\label{eq:saddle-asymptotic}
 \frac{R_{P;m,n}(z)}{m!}
 =\frac{e^{n\Xi_*(z,\alpha)}}{\sqrt n}
 \bigl(\mathfrak a(z,\alpha)+O(n^{-1})\bigr),
 \qquad \mathfrak a(z,\alpha)\ne0.
\end{equation}
The expansion is uniform on smaller compact neighborhoods, also when the monic degree-\(d\) polynomial varies in its coefficients. In particular, the Rodrigues descendants are zero-free on such compact parameter sets for all sufficiently large \(n\).
\end{proposition}
\begin{proof}
Cauchy's formula gives
\begin{equation}\label{eq:cauchy}
 \frac{R_{P;m,n}(z)}{m!}=
 \frac1{2\pi i}\oint_\gamma
 \exp\bigl(n[\log P(w)-\alpha\log(w-z)]\bigr)
 \frac{\dd w}{w-z}.
\end{equation}
At the exact ratio \(\alpha=m/n\) the integrand is single-valued, so the value of the integral is independent of the local logarithm branches. Shrink the parameter neighborhood and then take a compact subset of it. Persistence of the simple saddle and of the separating contour provides a family \(\gamma_\lambda\), where \(\lambda=(P,z,\alpha)\), with uniformly bounded length and a number \(c>0\) such that
\begin{equation}\label{eq:uniform-saddle-gap}
 L_\lambda(w)\le \Ree\Xi_\lambda-c
\end{equation}
away from one fixed-size saddle disk. The contour tubes and the saddle disks stay a positive distance from \(z\) and from every root lying on them. This remains true under coefficient perturbation, even if roots split elsewhere.

The holomorphic Morse lemma with parameters gives a coordinate \(t\mapsto W_\lambda(t)\), defined for \(|t|<2\delta\) with one common \(\delta>0\), such that
\[
 W_\lambda(0)=w_*(\lambda),\qquad
 \Phi_\lambda(W_\lambda(t))=\Xi_\lambda-t^2.
\]
Choose the sign of \(t\) so that a short oriented passage of \(\gamma_\lambda\) enters the negative descending sector and leaves the positive one. In the Morse coordinate, replace this passage by \([-\delta,\delta]\) and two straight endpoint connectors within those sectors; a local primitive preserves the integral. The parameter interval and sector margins can be chosen uniformly, and both the connectors and the complementary contour arc have a uniform strict height gap. Include the connectors in the remainder, reducing \(c\) if necessary. All derivatives of \(W_\lambda\) needed below are uniformly bounded, and \(|W_\lambda'(0)|\) and \(|W_\lambda(0)-z|\) are uniformly bounded away from zero. The local contribution to \eqref{eq:cauchy} is therefore
\[
 \frac{e^{n\Xi_\lambda}}{2\pi i}
 \int_{-\delta}^{\delta}e^{-nt^2}A_\lambda(t)\dd t,
 \qquad
 A_\lambda(t):=\frac{W_\lambda'(t)}{W_\lambda(t)-z}.
\]
Write \(A_\lambda(t)=A_\lambda(0)+A_{1,\lambda}t+t^2E_\lambda(t)\), with \(E_\lambda\) uniformly bounded. The odd term integrates to zero, and comparison with the full Gaussian integral gives, uniformly in \(\lambda\),
\[
 \int_{-\delta}^{\delta}e^{-nt^2}A_\lambda(t)\dd t
 =\sqrt{\frac\pi n}\left(A_\lambda(0)+O(n^{-1})\right).
\]
By \eqref{eq:uniform-saddle-gap}, the remainder of the contour is \(O(e^{n(\Ree\Xi_\lambda-c)})\) in absolute value, uniformly in all parameters; after the factor \(e^{n\Xi_\lambda}n^{-1/2}\) is removed, this is \(O(n^{-1})\).

Consequently \eqref{eq:saddle-asymptotic} holds with
\[
 \mathfrak a(\lambda)=\frac{\sqrt\pi}{2\pi i}
 \frac{W_\lambda'(0)}{w_*(\lambda)-z}.
\]
The orientation is incorporated in the chosen sign of \(t\), so the two local halves add rather than cancel. The displayed amplitude is nonzero and has a uniform positive lower bound on compact parameter sets. The parameter-dependent Morse lemma gives its stated regularity, and all constants above remain uniform for nearby monic degree-\(d\) coefficients. This proves both the expansion and the uniform zero-free conclusion.
\end{proof}

\begin{lemma}[Compactness of logarithmic potentials]\label{lem:pot-compact}
If probability measures \(\nu_n\) are supported in one compact set and \(\nu_n\Longrightarrow\nu\), then \(U_{\nu_n}\to U_\nu\) in \(L^1_{\mathrm{loc}}(\C)\). Their potentials are uniformly integrable on each bounded planar set. The same assertion holds for any family with that common support whose closure is weakly compact.
\end{lemma}
\begin{proof}
Truncate the kernel to \(\max\{\log|z-\xi|,\log\eps\}\). On compact products this is bounded and continuous, so the corresponding potentials converge in \(L^1_{\mathrm{loc}}\). The discarded part has \(L^1\) norm at most a constant times \(\eps^2(1+|\log\eps|)\), uniformly in \(\xi\). For uniform integrability, integrate \(|\log|z-\xi||\) on a set of small area and split at \(|z-\xi|=\eps\). The near part has the preceding uniform bound, and the far part is bounded by a constant times the area. Fubini's theorem proves the assertions.
\end{proof}

\begin{theorem}[The fixed-ratio limit]\label{thm:fixed}
Let \(P\) be monic of degree \(d\ge2\), and let \(0\le m_n<dn\) with \(m_n/n\to\alpha\in(0,d)\). Then
\[
 \mu_{P;m_n,n}^{\mathrm{root}}=\frac1{dn-m_n}\divz R_{P;m_n,n}
\]
converges weakly to a probability measure \(\mu_{\alpha,P}\) supported in \(\Hull_P\). The logarithmic potentials converge in \(L^1_{\mathrm{loc}}\) to
\begin{equation}\label{eq:fixed-potential-intro}
 U_{\mu_{\alpha,P}}(z)=
 \frac{\Blev_P(z,\alpha)-\mathfrak h_d(\alpha)}{d-\alpha}.
\end{equation}
At a root, the right-hand side means its subharmonic extension. The limit is independent of the integer sequence. More generally, if monic degree-\(d\) polynomials \(P_n\) converge in coefficients to \(P\), then the normalized root-counting measures of \((P_n^n)^{(m_n)}\) have the same limit \(\mu_{\alpha,P}\).
\end{theorem}
\begin{proof}
Gauss--Lucas gives a common compact support. For almost every \(z\), Lemma~\ref{lem:regular} applies at \((z,\alpha)\), and Proposition~\ref{prop:uniform-saddle} applies at \((z,m_n/n)\) for large \(n\). Therefore
\[
 \frac1n\log\left|\frac{R_{P;m_n,n}(z)}{m_n!}\right|
 \longrightarrow\Blev_P(z,\alpha).
\]
The leading coefficient of \(R_{P;m_n,n}/m_n!\) is \(\binom{dn}{m_n}\), and Stirling's formula gives
\[
 \frac1n\log\binom{dn}{m_n}\longrightarrow \mathfrak h_d(\alpha).
\]
After division by \((dn-m_n)/n\), the normalized logarithmic potentials converge almost everywhere to \eqref{eq:fixed-potential-intro}. Their root measures have one common compact support, so the uniform-integrability conclusion of Lemma~\ref{lem:pot-compact} applies without first assuming weak convergence. Vitali's theorem on each bounded disk therefore gives \(L^1_{\mathrm{loc}}\) convergence. Taking distributional Laplacians identifies every weakly convergent subsequence, and compactness gives convergence of the full sequence to a positive measure of mass one.

The subharmonic representative of the limiting potential supplies its extension at the roots. Off the roots, Lemma~\ref{lem:barrier-cont} gives a continuous expression agreeing almost everywhere with that potential. The submean characterization of the subharmonic representative gives equality everywhere there.

For varying polynomials \(P_n\), their root convex hulls still lie in one compact set. At almost every limiting observation point, the selected simple saddle and its separating contour for \((P,z,\alpha)\) persist at \((P_n,z,m_n/n)\) by Proposition~\ref{prop:uniform-saddle}. The same almost-everywhere and compactness argument applies. Root splitting may change the degree of the reduced saddle equation, but it does not affect this selected saddle or the contour tube; any additional saddle branches created near the split roots are irrelevant to the local asymptotic.
\end{proof}

\subsection{Atoms and the selected inverse branch}
\begin{proposition}[The fixed-ratio atoms]\label{prop:atoms}
For \(x\in\R\), write \(x_+=\max\{x,0\}\). The atomic part of the limiting measure \(\mu_{\alpha,P}\) is
\begin{equation}\label{eq:fixed-atoms}
 \mu_{\alpha,P}^{\mathrm{at}}=
 \sum_{j=1}^s\frac{(r_j-\alpha)_+}{d-\alpha}\delta_{a_j}.
\end{equation}
When \(\alpha<r_j\), the root \(a_j\) is isolated in \(\supp\mu_{\alpha,P}\). The potential is continuous off \(Z(P)\). Wherever one saddle branch realizes the barrier, the potential is that branch's critical height minus \(\mathfrak h_d(\alpha)\), divided by \(d-\alpha\).
\end{proposition}
\begin{proof}
Continuity off the roots follows from Lemma~\ref{lem:barrier-cont} and rules out atoms there. Fix a root \(a\) of multiplicity \(r\). For \(m<rn\), its multiplicity in \(R_{P;m,n}\) is exactly \(rn-m\). Thus if \(\alpha<r\), weak convergence forces an atom of mass at least \((r-\alpha)/(d-\alpha)\).

Since \(h_P(w)=(w-a)/r+O((w-a)^2)\), the inverse branch of \(F_{P,\alpha}\) through \(a\), for \(\alpha\ne r\), satisfies
\[
 w-a=\frac{r}{r-\alpha}(z-a)+O((z-a)^2),\qquad
 w-z=\frac{\alpha}{r-\alpha}(z-a)+O((z-a)^2).
\]
Its critical height is \((r-\alpha)\log|z-a|\) plus a harmonic function; all other heights remain bounded near \(a\). For \(\alpha<r\), the forced atom makes the limiting potential tend uniformly to \(-\infty\) as \(z\to a\), because the potential of the remaining positive measure is locally bounded above. Hence the root branch is selected on a punctured neighborhood. This gives the exact mass and isolation.

For \(\alpha>r\), the root branch tends to \(+\infty\), whereas a fixed contour around \(a\) bounds the barrier from above. The root branch is therefore unselected, and the selected heights remain bounded. At \(\alpha=r<d\), the one-root case \(P(w)=(w-a)^d\) cannot occur. The order of vanishing of \(F_{P,r}(w)-a\) at \(w=a\) (the local mapping degree) is therefore some \(\mathfrak e_a\ge2\). On the corresponding inverse branches, \(w-a\) has order \((z-a)^{1/\mathfrak e_a}\) and \((w-z)/(w-a)\to1\). The logarithmic terms cancel, so the critical heights are bounded here as well. There is no atom at the threshold \(\alpha=r\), completing \eqref{eq:fixed-atoms}.
\end{proof}

\begin{proposition}[The selected branch of \(w_\alpha(z)\) off the support]\label{prop:selected-C}
Set
\[
 C_{\alpha,P}(z):=C_{\mu_{\alpha,P}}(z),\qquad
 \widehat C_{\alpha,P}(z):=(d-\alpha)C_{\alpha,P}(z).
\]
The identity
\begin{equation}\label{eq:reduced-symbol}
 \widehat C_{\alpha,P}(z)^s
 Q_P\left(z+\frac{\alpha}{\widehat C_{\alpha,P}(z)}\right)
 -\widehat C_{\alpha,P}(z)^{s-1}
 N_P\left(z+\frac{\alpha}{\widehat C_{\alpha,P}(z)}\right)=0
\end{equation}
holds almost everywhere in the plane and everywhere off \(\supp\mu_{\alpha,P}\). After cancellation of the displayed denominators, the left-hand side is a polynomial in \(\widehat C_{\alpha,P}(z)\) with constant term \(\alpha^{s-1}(\alpha-d)\). Thus \(C_{\alpha,P}\) has no zeros off its support.

On each component of \(\C\setminus\supp\mu_{\alpha,P}\),
\begin{equation}\label{eq:w-selected}
 w_\alpha(z):=z+\frac{\alpha}{\widehat C_{\alpha,P}(z)}
\end{equation}
is the holomorphic inverse branch selected by the barrier, and
\begin{equation}\label{eq:C-selected}
 F_{P,\alpha}(w_\alpha(z))=z,
 \qquad \widehat C_{\alpha,P}(z)=\frac{P'(w_\alpha(z))}{P(w_\alpha(z))},
 \qquad F_{P,\alpha}'(w_\alpha(z))\ne0.
\end{equation}
The last derivative is with respect to \(w\). At a point of \(Z(P)\setminus\supp\mu_{\alpha,P}\), selection means the unique holomorphic continuation of the barrier-selected branch from a punctured neighborhood.
\end{proposition}
\begin{proof}
On a regular chamber, differentiate the selected critical height in \eqref{eq:fixed-potential-intro}. The identities \(2\partial_z\Ree\Xi_i=\partial_z\Xi_i\) and \eqref{eq:phase-z} give \(\widehat C_{\alpha,P}(z)=\alpha/(w_\alpha(z)-z)\). Substitution into the reduced saddle equation yields \eqref{eq:reduced-symbol}. The exceptional set has planar measure zero. Off the support, the resulting polynomial identity is holomorphic and hence holds everywhere. Its nonzero constant term proves nonvanishing of the Cauchy transform.

The branch cannot meet a root of \(Q_P\), since \(\widehat C_{\alpha,P}(z)Q_P(w_\alpha(z))-N_P(w_\alpha(z))=0\) and \(Q_P,N_P\) are coprime. Division gives the first two identities in \eqref{eq:C-selected}; differentiating \(F_{P,\alpha}(w_\alpha(z))=z\) gives the last.
\end{proof}

\subsection{Continuity and root thresholds}
\begin{lemma}[Joint weak continuity and uniformity]\label{lem:alpha-cont}
In fixed degree \(d\), \((P,\alpha)\mapsto\mu_{\alpha,P}\) is jointly weakly continuous for \(0<\alpha<d\). For every compact family \(\mathscr P\) of monic degree-\(d\) polynomials, every compact interval \(I\subset(0,d)\), and every continuous function \(\varphi:\C\to\C\),
\[
 \sup_{\substack{P\in\mathscr P,\ 0\le m<dn\\ m/n\in I}}
 \left|\int\varphi\dd\mu_{P;m,n}^{\mathrm{root}}
       -\int\varphi\dd\mu_{m/n,P}\right|\longrightarrow0.
\]
In particular these assertions hold with \(P\) fixed and \(\varphi\) continuous on \(\Hull_P\).
\end{lemma}
\begin{proof}
Let \((P_j,\alpha_j)\to(P,\alpha)\). Their root hulls lie in one compact disk, on which weak convergence of probability measures is metrizable. By the fixed-polynomial statement of Theorem~\ref{thm:fixed}, choose increasing \(n_j\) such that the normalized root-counting measure of \((P_j^{n_j})^{(\lfloor\alpha_jn_j\rfloor)}\) is within \(1/j\) of \(\mu_{\alpha_j,P_j}\) in that metric. Its derivative ratio tends to \(\alpha\). The varying-polynomial statement of the same theorem therefore gives convergence to \(\mu_{\alpha,P}\) and proves joint continuity.

If the uniform assertion failed, choose a sequence of coefficients and ratios in \(\mathscr P\times I\) with positive discrepancy. A convergent subsequence contradicts Theorem~\ref{thm:fixed} and the continuity just proved. For a fixed \(P\), only the restriction of \(\varphi\) to \(\Hull_P\) is relevant, giving the last assertion.
\end{proof}

A weak limit can miss a bounded number of zeros. Away from the roots, the harmonic-barrier lemma excludes this possibility. At a root, we must first understand the ratio where its atom vanishes.

\begin{lemma}[The threshold root stays in the support]\label{lem:threshold-root}
If \(a\) has multiplicity \(r<d\), then
\[
 a\in\supp\mu_{\alpha,P}\qquad(0<\alpha\le r).
\]
At \(\alpha=r\), the atom therefore disappears into the remaining support, not across a support-free neighborhood.
\end{lemma}
\begin{proof}
By Proposition~\ref{prop:atoms}, only \(\alpha=r\) remains. Suppose a closed disk \(\overline D\) centered at \(a\), containing no other root, is disjoint from \(\supp\mu_{r,P}\). Choose a closed annulus \(A\) containing \(\partial D\) in its interior and still disjoint from \(\supp\mu_{r,P}\cup Z(P)\). On a neighborhood of \(A\), the potential and therefore the barrier are harmonic at \(\beta=r\).

For every \(z\in A\), Lemma~\ref{lem:harmonic-gate} supplies a unique simple selected saddle and a persistent one-saddle contour. A finite subcover of the compact annulus gives one number \(\eta>0\) for which these contours persist on all of \(A\) whenever \(|\beta-r|<\eta\). Lemma~\ref{lem:multiple-gates} then identifies the barrier locally with the real part of one holomorphic critical-value branch. Hence the fixed-ratio potential is harmonic on \(A\), so \(A\cap\supp\mu_{\beta,P}=\varnothing\), throughout this common \(\beta\)-window. On overlaps, the locally selected expressions
\[
 \widehat C_{\beta,P}(z)=\frac{\beta}{w_\beta(z)-z}
\]
agree because they are the Cauchy transform of the same measure. Thus \(\widehat C_{\beta,P}\) is jointly holomorphic in \(z\) and real-analytic in \(\beta\) on a neighborhood of \(\partial D\).

Consequently
\[
 M(\beta):=\frac{1}{2\pi i}\oint_{\partial D}
             \widehat C_{\beta,P}(z)\dd z
          =(d-\beta)\mu_{\beta,P}(D)
\]
is differentiable. It is nonnegative and \(M(r)=0\), so \(M'(r)=0\). The atom formula, however, gives \(M(r-h)\ge h\) for all sufficiently small \(h>0\), contradicting differentiability at \(r\).
\end{proof}

\begin{lemma}[Separation beyond a root threshold]\label{lem:merged-end}
Let \(a\) have multiplicity \(r\), let \(r<\alpha<d\), and assume \(a\notin\supp\mu_{\alpha,P}\). At the coalesced observation point \(a\), the landscape
\[
 w\longmapsto\log|P(w)|-\alpha\log|w-a|
\]
has a one-saddle separating contour winding once about \(a\). Moreover, there are neighborhoods \(\mathscr U\) of \(P\) in the monic coefficient space, \(V\) of \(a\), and \(I\Subset(r,d)\) of \(\alpha\) such that, for every \((\widetilde P,z,\beta)\in\mathscr U\times V\times I\), the corresponding landscape has a one-saddle separating contour with a uniform strict gap. This remains true when \(z\) is a root of \(\widetilde P\). At exact ratios \(\beta=m/n\), the Cauchy integral has the uniform nonzero asymptotic of Proposition~\ref{prop:uniform-saddle}.
\end{lemma}
\begin{proof}
Write \(P(w)=(w-a)^rG(w)\), where \(G(a)\ne0\). At observation point \(a\), the coalesced landscape is
\begin{equation}\label{eq:coalesced-landscape}
 L^{\mathrm c}(w)=(r-\alpha)\log|w-a|+\log|G(w)|.
\end{equation}
This is the logarithmic landscape for the polynomial \(G\), the observation point \(a\), and the ratio \(\alpha-r\in(0,d-r)\). The barrier arguments of Section~\ref{sec:barrier} use only this strict ratio inequality and apply verbatim also when \(\deg G=1\). Thus they apply to \eqref{eq:coalesced-landscape}; denote its barrier by \(B^{\mathrm c}\).

We first identify \(B^{\mathrm c}\) with the limiting barrier as the observation point approaches \(a\). A fixed contour that is near-minimizing for \(B^{\mathrm c}\) also winds around every sufficiently nearby \(z\), and its landscape converges uniformly, giving
\[
 \limsup_{z\to a}\Blev_P(z,\alpha)\le B^{\mathrm c}.
\]
For the reverse inequality, take a path for the coalesced landscape joining \(a\) to infinity and lying above \(B^{\mathrm c}-\eps\), and truncate it on a circle \(|w-a|=\delta\). On the compact truncated part, the moving landscapes converge uniformly to the coalesced one. On the circle and on a radial attachment from \(z\) to the circle,
\[
 \log|P(w)|-\alpha\log|w-z|
 \ge C+(r-\alpha)\log\delta,
\]
uniformly for \(|z-a|<\delta/4\); here \(C\) is independent of small \(\delta\). Since \(r-\alpha<0\), this lower bound tends to \(+\infty\) as \(\delta\downarrow0\). The attachment therefore gives \(\liminf_{z\to a}\Blev_P(z,\alpha)\ge B^{\mathrm c}\). Hence
\begin{equation}\label{eq:coalesced-barrier-limit}
 \lim_{z\to a}\Blev_P(z,\alpha)=B^{\mathrm c}.
\end{equation}

Because \(a\notin\supp\mu_{\alpha,P}\), the fixed-ratio potential is harmonic on a disk about \(a\). The potential formula and \eqref{eq:coalesced-barrier-limit} therefore extend \(\Blev_P(\cdot,\alpha)\) to a harmonic function \(H\) on that disk, with \(H(a)=B^{\mathrm c}\). Let \(w_0\) be any connecting saddle of the coalesced landscape. Choose an ascending path through \(w_0\) joining the two positive ends and having its unique minimum \(B^{\mathrm c}\) at \(w_0\). The preceding attachment construction turns its truncated part into a path for nearby observation points. Its minimum is a lower touching function for \(H\), and expansion in \(z=a+\zeta\) gives
\[
 H(a+\zeta)\ge B^{\mathrm c}
   +\Ree\!\left(\frac{\alpha}{w_0-a}\zeta\right)+o(|\zeta|).
\]
Applying this in the directions \(\zeta\) and \(-\zeta\) yields
\begin{equation}\label{eq:coalesced-gradient}
 2\partial_zH(a)=\frac{\alpha}{w_0-a}.
\end{equation}
Distinct connecting saddles would give distinct right-hand sides, so the coalesced landscape has exactly one connecting saddle.

On the punctured disk, Lemma~\ref{lem:harmonic-gate} gives
\[
 w(z)=z+\frac{\alpha}{2\partial_zH(z)},
 \qquad F_{P,\alpha}(w(z))=z.
\]
The first formula extends holomorphically to \(z=a\) by \eqref{eq:coalesced-gradient}; the second then extends by the identity theorem. Differentiating at \(a\) gives \(F_{P,\alpha}'(w(a))w'(a)=1\), so the unique coalesced connecting saddle is simple. Lemma~\ref{lem:arbitrary-contour}, applied to \eqref{eq:coalesced-landscape}, now gives a one-saddle separating contour with a strict gap.

It remains to prove uniform persistence when the multiplicity-\(r\) root splits. Choose an open interval \(I\) with \(\alpha\in I\Subset(r,d)\) and put \(\vartheta_r=\pi/(2r+2)\) and \(c_r=\sin\vartheta_r\). Choose \(\delta>0\) so small that \(a\) is the only root of \(P\) in \(|w-a|\le4\delta\), the base barrier contour is disjoint from \(|w-a|\le2\delta\), and \(|G(w)|\ge2g\) there for some \(g>0\). Since \(\inf_{\beta\in I}(\beta-r)>0\), decrease \(\delta\) further so that, for every \(\beta\in I\), both
\begin{align}
 \log g+r\log c_r+(r-\beta)\log(2\delta),
 \label{eq:split-root-ray-bound}\\
 \log g+r\log(7\delta/8)-\beta\log(9\delta/8)
 \label{eq:split-root-circle-bound}
\end{align}
exceed \(B^{\mathrm c}+3\).

After shrinking the coefficient neighborhood, every \(\widetilde P\in\mathscr U\) has exactly \(r\) roots \(b_1,\ldots,b_r\), counted with multiplicity, in \(|w-a|<\delta/8\) and can be factored on \(|w-a|\le2\delta\) as
\[
 \widetilde P(w)=\widetilde G(w)\prod_{j=1}^r(w-b_j),
 \qquad |\widetilde G(w)|\ge g.
\]
For \(z\in V:=\{|z-a|<\delta/8\}\), remove from the unit circle the \(r\) arcs of angular radius \(\vartheta_r\) centered at the directions of the nonzero vectors \(b_j-z\). Their total length is less than \(2\pi\), so a remaining unit vector \(u\) exists. Elementary minimization in \(t>0\) then gives
\[
 |z+tu-b_j|\ge c_r t\qquad(1\le j\le r),
\]
with the same inequality when \(b_j=z\). Along this ray from \(z\) until it meets \(|w-a|=\delta\), one has \(t\le2\delta\) and therefore
\begin{equation}\label{eq:split-root-attachment}
 \log|\widetilde P(z+tu)|-\beta\log t
 \ge \log g+r\log c_r+(r-\beta)\log(2\delta).
\end{equation}
On the whole circle \(|w-a|=\delta\) one has \(|w-b_j|\ge7\delta/8\) and \(|w-z|\le9\delta/8\), which gives the lower bound in \eqref{eq:split-root-circle-bound}.

The base contour and an ascending path through its saddle have compact parts outside this circle that avoid all roots. Under coefficient, observation-point, and ratio perturbations, the simple saddle continues by the implicit-function theorem. Replace the local contour segment by its continued descending axis and the local path segment by its continued ascending axis. Shrink the parameter neighborhoods so that the continued saddle height differs from \(B^{\mathrm c}\) by less than one. Uniform convergence on the remaining compact parts, together with \eqref{eq:split-root-attachment} and \eqref{eq:split-root-circle-bound}, gives an upper contour and a lower joining path with the same unique saddle height and with uniform strict gaps away from the saddle. The min--max formula therefore identifies that height with the perturbed barrier. The strict superlevel component containing \(z\) lies inside the contour and the component containing infinity lies outside it; their closures can meet only on the contour, where the continued saddle is the sole point at barrier height. It is therefore the unique connecting saddle. Proposition~\ref{prop:uniform-saddle} now gives the asserted uniform nonzero asymptotic, irrespective of root splitting or of a collision between \(z\) and a perturbed root.
\end{proof}

\subsection{Hausdorff convergence at a fixed ratio \texorpdfstring{\(\alpha\in(0,d)\)}{alpha}}
\label{subsec:fixed-sets}
The preceding lemmas now give convergence of the complete zero sets, uniformly through root and critical-point collisions.

\begin{theorem}[Fixed-ratio Hausdorff convergence]\label{thm:fixed-Hausdorff}
Let \(P_j\to P\) be a sequence of monic polynomials of degree \(d\) converging to \(P\), let \(n_j\to\infty\), and let \(m_j/n_j\to\alpha\in(0,d)\). Then
\begin{equation}\label{eq:fixed-Hausdorff}
 d_{\mathrm H}\!\left(Z((P_j^{n_j})^{(m_j)}),
                  \supp\mu_{\alpha,P}\right)\longrightarrow0.
\end{equation}
The map \((P,\alpha)\mapsto\supp\mu_{\alpha,P}\) is Hausdorff-continuous. For every compact coefficient set \(\mathscr P\) of monic degree-\(d\) polynomials and compact interval \(I\subset(0,d)\),
\begin{equation}\label{eq:fixed-uniform-Hausdorff}
 \sup_{\substack{P\in\mathscr P,\ 0\le m<dn\\m/n\in I}}
 d_{\mathrm H}\!\left(Z((P^n)^{(m)}),\supp\mu_{m/n,P}\right)
 \longrightarrow0.
\end{equation}
These statements include integer ratios, multiple roots, and multiple critical points of \(P\).
\end{theorem}
\begin{proof}
Weak convergence in Theorem~\ref{thm:fixed} gives the lower Hausdorff estimate. For the upper estimate, suppose descendant zeros approach \(z\notin\supp\mu_{\alpha,P}\). If \(P(z)\ne0\), Lemma~\ref{lem:harmonic-gate} and Proposition~\ref{prop:uniform-saddle} exclude zeros near \(z\) for all large \(j\). If \(z=a\) is a root of multiplicity \(r\), Proposition~\ref{prop:atoms} and Lemma~\ref{lem:threshold-root} give \(\alpha>r\), and Lemma~\ref{lem:merged-end} gives the same exclusion. All zeros lie in a common compact root hull, proving \eqref{eq:fixed-Hausdorff}.

For continuity, let \((P_j,\alpha_j)\to(P,\alpha)\) and choose increasing \(n_j\) such that the fixed-polynomial conclusion gives
\[
 d_{\mathrm H}\!\left(Z((P_j^{n_j})^{(\lfloor\alpha_jn_j\rfloor)}),
                         \supp\mu_{\alpha_j,P_j}\right)<1/j.
\]
Apply \eqref{eq:fixed-Hausdorff} to this diagonal sequence. Finally, a failure of \eqref{eq:fixed-uniform-Hausdorff} would supply a sequence in a compact parameter set. A convergent subsequence contradicts \eqref{eq:fixed-Hausdorff} and the continuity just established.
\end{proof}

Figure~\ref{fig:fixed-ratio} illustrates Theorems~\ref{thm:fixed} and~\ref{thm:fixed-Hausdorff} for the quartic of Figure~\ref{fig:landscape}. For \(\alpha<1\) each root is an atom of mass \((1-\alpha)/(d-\alpha)\) (Proposition~\ref{prop:atoms}), and the rest of the support consists of arcs near the critical points of \(P\). For larger \(\alpha\) the support in this example is a tree. Its endpoints are critical values of \(F_{P,\alpha}\), that is, points where two saddles collide.

\begin{figure}[htbp]
\centering
\includegraphics[width=0.49\textwidth]{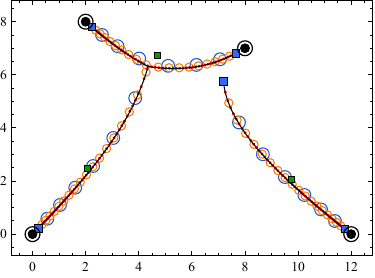}\hfill
\includegraphics[width=0.49\textwidth]{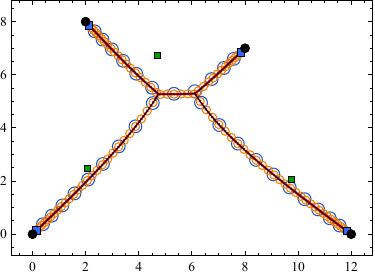}
\caption{The zeros of \((P^n)^{(m)}\) with \(m/n=\alpha\) for \(P(z)=z(z-12)(z-2-8i)(z-8-7i)\) and \(n=12\) (blue circles), \(n=36\) (orange circles) and \(n=108\) (red dots), together with the limiting support \(\supp\mu_{\alpha,P}\) (black). Left: \(\alpha=\frac12\); the ringed roots are atoms of \(\mu_{\alpha,P}\) of mass \(\frac17\) each, and the multiple zeros at the roots are not drawn. Right: \(\alpha=\frac32\). Blue squares: the endpoints of \(\supp\mu_{\alpha,P}\); green squares: the critical points of \(P\).}
\label{fig:fixed-ratio}
\end{figure}

\begin{proof}[Proof of Theorem~\ref{thm:intro-fixed-ratio}]
The weak convergence, the independence of the approximating integer sequence, and stability under coefficient perturbations are Theorem~\ref{thm:fixed}. Joint weak continuity and uniformity of the measure convergence are Lemma~\ref{lem:alpha-cont}. The Hausdorff convergence of the complete zero sets, joint continuity of their limiting supports, and the corresponding uniformity are Theorem~\ref{thm:fixed-Hausdorff}.
\end{proof}

\section{Endpoint ratios}\label{sec:endpoints}
The asymptotic in Section~\ref{sec:limits} does not cover the endpoint regimes. Crude considerations show that if \(1\le m_n=o(n)\), then the zero sets should approach \(Z(P)\cup Z(P')\), whereas \(m_n=0\) gives exactly \(Z(P)\). Similarly, if \(dn-m_n=o(n)\), the zero sets should approach the centroid of the root multiset. We treat these endpoint regimes by a uniform small-order coefficient estimate.

\subsection{A uniform coefficient estimate}

\begin{lemma}[Small-order coefficients]\label{lem:small-coeff}
Let \(\mathscr A(z,t)\) be jointly holomorphic on a neighborhood of \(X\times\{t\in\C:|t|\le r\}\), where \(r>0\) and \(X\subset\C\) is compact, and write \(\mathscr A_z(t)=\mathscr A(z,t)\). Suppose that \(|\mathscr A(z,0)|\) and \(|\mathfrak b(z)|\), where \(\mathfrak b(z):=\frac{\partial_t\mathscr A(z,0)}{\mathscr A(z,0)}\), are bounded below by positive constants on \(X\). If \(k_n\ge1\) and \(k_n/n\to0\), then
\[
 [t^{k_n}]\mathscr A_z(t)^n\ne0\qquad(z\in X)
\]
for all sufficiently large \(n\). This remains true for jointly holomorphic functions \(\mathscr A_n(z,t)\), with \(\mathscr A_{n,z}(t):=\mathscr A_n(z,t)\), that are uniformly bounded on a common neighborhood of \(X\times\{t\in\C:|t|\le r\}\), provided both lower bounds are uniform in \(n\). On a smaller parameter neighborhood the coefficients also satisfy the uniform logarithmic-derivative estimate
\begin{equation}\label{eq:small-coeff-logder}
 \partial_z\log\bigl([t^{k_n}]\mathscr A_{n,z}(t)^n\bigr)
 =n\,\partial_z\log\mathscr A_{n,z}(0)+O(k_n).
\end{equation}
Here \(\partial_z\log f\) means \((\partial_z f)/f\), so no global logarithm is required. The estimate is valid for \(k_n=0\) as an exact identity.
\end{lemma}
\begin{proof}
After enlarging \(X\) slightly inside the holomorphy neighborhood, Cauchy estimates and the uniform lower bound for \(|\mathscr A(z,0)|\) give a number \(\rho>0\) such that all functions under consideration are zero-free for \(|t|\le\rho\). On finitely many parameter charts choose
\[
 \varpi_z(t):=\log\frac{\mathscr A_z(t)}{\mathscr A_z(0)}.
\]
For a varying family, the normalized logarithms and all derivatives needed below are uniformly bounded on a common smaller polydisk. The quantity \(\mathfrak b(z)=\varpi_z'(0)\) is uniformly bounded away from zero. Thus the family index may be included among the parameters. Write \(k=k_n\).

If \(k\) remains bounded, direct coefficient expansion, uniformly on the parameter charts, gives
\[
 [t^k]\mathscr A_z(t)^n
 =\mathscr A_z(0)^n\frac{n^k\mathfrak b(z)^k}{k!}
  \bigl(1+O(n^{-1})\bigr).
\]
Now suppose \(k\to\infty\) and put \(\eps=k/n\to0\). The implicit-function theorem, uniformly in the parameters, gives a unique solution of
\[
 t\varpi_z'(t)=\eps,
 \qquad
 t_*=\frac{\eps}{\mathfrak b(z)}+O(\eps^2),
\]
in \(|t|<\rho/2\). On the coefficient circle \(t=t_*e^{i\theta}\), \(-\pi\le\theta\le\pi\), the exponent relative to its value at \(\theta=0\) is
\[
 k(e^{i\theta}-1-i\theta)
 +n\sum_{j\ge2}c_j(z)t_*^j
 \bigl(e^{ij\theta}-1-j(e^{i\theta}-1)\bigr),
\]
where \(\varpi_z(t)=\sum_{j\ge1}c_j(z)t^j\). With \(x=e^{i\theta}\),
\[
 x^j-1-j(x-1)=(x-1)^2
       \sum_{q=0}^{j-2}(j-1-q)x^q.
\]
Since \(|x-1|^2=2(1-\cos\theta)\), this proves the uniform estimate
\begin{equation}\label{eq:coefficient-circle-bound}
 \bigl|e^{ij\theta}-1-j(e^{i\theta}-1)\bigr|
 \le j^2(1-\cos\theta),
 \qquad j\ge2,
\end{equation}
for \(-\pi\le\theta\le\pi\). Cauchy bounds for \(c_j\) and \(|t_*|=O(\eps)\) therefore bound the real part of the second sum by \(Ck\eps(1-\cos\theta)\). For small \(\eps\) the full real part is at most \(-\tfrac12k(1-\cos\theta)\), uniformly on the entire circle. Its quadratic coefficient at zero is
\[
 -\frac n2\bigl(t_*\varpi_z'(t_*)+t_*^2\varpi_z''(t_*)\bigr)
 =-\frac k2\bigl(1+O(\eps)\bigr).
\]
The third through sixth derivatives of the exponent are \(O(k)\) uniformly on a fixed neighborhood of zero. After the scaling \(\theta=s/\sqrt{k}\), its cubic correction is odd and integrates to zero on the symmetric contour; the quartic term and the square of the cubic term are \(O(k^{-1})\) after Gaussian integration. On \(k^{-2/5}\le|\theta|\le\pi\), the preceding global bound is \(O(e^{-c k^{1/5}})\). Thus the uniform saddle estimate is
\[
 [t^k]\mathscr A_z(t)^n=
 \frac{\mathscr A_z(t_*)^n t_*^{-k}}
 {\sqrt{2\pi n\bigl(t_*\varpi_z'(t_*)+t_*^2\varpi_z''(t_*)\bigr)}}
 \bigl(1+O(k^{-1})\bigr).
\]
The radicand is \(2\pi k(1+O(\eps))\), so the square root is chosen coherently near the positive real square root and is uniformly nonzero. All other displayed factors are nonzero. Every sequence has either a bounded-\(k\) subsequence or a subsequence with \(k\to\infty\), which proves the asserted eventual nonvanishing.

For \eqref{eq:small-coeff-logder}, carry out the same estimates on a slightly larger complex parameter neighborhood. The relative errors are holomorphic there and uniformly \(O(n^{-1})\) in the bounded-\(k\) case and \(O(k^{-1})\) in the saddle case, so Cauchy estimates control their parameter derivatives. In the saddle case, differentiation of
\[
 n\log\mathscr A_z(t_*)-k\log t_*
\]
produces no \(\partial_z t_*\) term because \(nt_*\varpi_z'(t_*)=k\). The remaining phase derivative beyond \(n\partial_z\log\mathscr A_z(0)\) is \(n\partial_z\varpi_z(t_*)=O(n|t_*|)=O(k)\), and the logarithmic derivative of the radicand is bounded. This proves \eqref{eq:small-coeff-logder}, uniformly also for varying families. The case \(k=0\) is the stated exact identity.
\end{proof}

\subsection{Endpoint exclusion and exact cluster counts}
\begin{theorem}[Endpoint exclusion]\label{thm:endpoint-zeros}
Let \(0\le m_n<dn\) and let \(z_n\in Z(R_{P;m_n,n})\) with \(z_n\to z\). Then
\begin{align}
 m_n/n\to0&\quad\Longrightarrow\quad
 z\in Z(P)\cup Z(P'),\label{eq:low-end}\\
 m_n/n\to d&\quad\Longrightarrow\quad
 z=\centroid_P,\label{eq:high-end}
\end{align}
where \(\centroid_P\) denotes the centroid of the root multiset of \(P\).
\end{theorem}
\begin{proof}
On a compact set avoiding \(Z(P)\cup Z(P')\), apply Lemma~\ref{lem:small-coeff} to \(\mathscr A_z(t)=P(z+t)\). Its constant term and logarithmic linear coefficient are bounded away from zero, and
\[
 [t^m]\mathscr A_z(t)^n=\frac{R_{P;m,n}(z)}{m!}.
\]
Order \(m=0\) is immediate. This proves \eqref{eq:low-end}.

For the upper endpoint, put \(k=dn-m\ge1\) and reverse the coefficient polynomial:
\[
 \mathscr A_z(t):=t^dP(z+t^{-1}).
\]
It has constant term \(1\) and linear coefficient \(d(z-\centroid_P)\). Furthermore,
\[
 [t^k]\mathscr A_z(t)^n=[u^{dn-k}]P(z+u)^n
 =\frac{R_{P;dn-k,n}(z)}{(dn-k)!}.
\]
Apply Lemma~\ref{lem:small-coeff} on compact sets avoiding the centroid. Gauss--Lucas confines all zeros to a common compact set, so these local zero-free statements apply to every convergent zero sequence and give \eqref{eq:high-end}.
\end{proof}

\begin{theorem}[Exact cluster counts for sublinear orders]
\label{thm:sublinear-counts}
Let \(P_n\to P\) be a sequence of monic polynomials of degree \(d\), and let \(M_n\ge0\) be integers with \(M_n/n\to0\). Write \(a_j,r_j\) as in \eqref{eq:root-data}. Let \(c_1,\ldots,c_t\) be the distinct points of \(Z(P')\setminus Z(P)\), and put \(\varkappa_i=\operatorname{ord}_{c_i}P'\). Choose pairwise disjoint open disks \(D_j\) about \(a_j\) and \(E_i\) about \(c_i\), small enough that their closures contain no other point of \(Z(P)\cup Z(P')\). For all sufficiently large \(n\), simultaneously for every \(0\le m\le M_n\),
\begin{equation}\label{eq:sublinear-counts}
 (\divz (P_n^n)^{(m)})(D_j)=nr_j-m,
 \qquad
 (\divz (P_n^n)^{(m)})(E_i)=m\varkappa_i.
\end{equation}
All zeros lie in these disks. No assumption is made on the separation or multiplicity of the roots of \(P_n\) inside a disk.

Consequently, for \(1\le m_n<dn\) with \(m_n/n\to0\),
\begin{equation}\label{eq:complete-low-endpoint}
 Z((P_n^n)^{(m_n)})\longrightarrow Z(P)\cup Z(P')
 \quad\text{in Hausdorff distance}.
\end{equation}
For \(0\le m_n<dn\) with \(m_n/n\to d\), the corresponding limit is \(\{\centroid_P\}\).
\end{theorem}
\begin{proof}
On a neighborhood of the union of the disk boundaries, both \(P\) and \(P'/P\) are nonzero. Apply the logarithmic-derivative estimate \eqref{eq:small-coeff-logder} with \(\mathscr A_{n,z}(t)=P_n(z+t)\). Writing \(R_{P_n;m,n}=(P_n^n)^{(m)}\), we obtain, uniformly there and for \(0\le m\le M_n\),
\begin{equation}\label{eq:sublinear-logder}
 \frac1n\frac{(R_{P_n;m,n})'}{R_{P_n;m,n}}
 =\frac{P_n'}{P_n}+O(m/n)
 \longrightarrow\frac{P'}P.
\end{equation}
Nonvanishing on these neighborhoods is part of the same lemma; at \(m=0\) it follows directly from \(P_n\to P\). To pass from the sequential lemma to the simultaneous range, argue by contradiction: any failure would produce indices \(n_j\) and orders \(m_j\le M_{n_j}\) with \(m_j/n_j\to0\), contradicting the lemma for that sequence. The same argument gives the uniform estimate in \eqref{eq:sublinear-logder}.

For \(1\le m\le M_n\), the meromorphic function \(R_{P_n;m,n}/R_{P_n;m-1,n}\), divided by \(n\), is uniformly close on each boundary to \(P'/P\). Its winding number about zero is therefore that of \(P'/P\). At \(a_j\) the rational function \(P'/P\) has a simple pole, even when \(r_j>1\), so this number is \(-1\). At \(c_i\) it is \(\varkappa_i\). Thus each differentiation changes the local zero count by \(-1\) in \(D_j\) and by \(\varkappa_i\) in \(E_i\). At order zero the counts are \(nr_j\) and zero, respectively, by coefficient convergence. Telescoping proves \eqref{eq:sublinear-counts}. Since \(\sum_{i=1}^t\varkappa_i=s-1\), the total count is
\[
 \sum_{j=1}^s(nr_j-m)+\sum_{i=1}^t m\varkappa_i
 =dn-sm+m(s-1)=dn-m,
\]
the full degree. There are no remaining zeros.

The disks can be arbitrarily small. For \(m_n\ge1\) all critical-point counts are positive, and \(nr_j-m_n>0\) for large \(n\). This proves \eqref{eq:complete-low-endpoint}. At the upper endpoint the reverse coefficient argument in Theorem~\ref{thm:endpoint-zeros} is uniform for \(P_n\to P\): its constant term is \(1\) and its linear coefficient is \(d(z-\centroid_{P_n})\). It excludes every compact set away from \(\centroid_P\). Gauss--Lucas gives a common compact enclosure and the zero sets are nonempty, proving the asserted Hausdorff limit.
\end{proof}

\begin{remark}[Order zero is an endpoint exception]
The lower endpoint in \eqref{eq:complete-low-endpoint} requires at least one differentiation. At order zero the limit is \(Z(P)\), not in general \(Z(P)\cup Z(P')\). A sequence alternating between order zero and positive sublinear orders need not have a single zero-set limit. The critical clusters in \eqref{eq:sublinear-counts} explain exactly why weak root-measure convergence misses this distinction.
\end{remark}

\subsection{Completing the family of limiting supports}
\begin{theorem}[Endpoint limits of the supports]\label{thm:endpoint-support}
In Hausdorff distance,
\begin{equation}\label{eq:endpoint-supports}
 d_{\mathrm H}\bigl(\supp\mu_{\alpha,P},Z(P)\cup Z(P')\bigr)\longrightarrow0
 \quad(\alpha\downarrow0),
 \qquad
 d_{\mathrm H}\bigl(\supp\mu_{\alpha,P},\{\centroid_P\}\bigr)\longrightarrow0
 \quad(\alpha\uparrow d).
\end{equation}
In particular,
\begin{equation}\label{eq:distinguished}
 Z(P)\cup Z(P')\cup\{\centroid_P\}\subseteq\Sh_P.
\end{equation}
\end{theorem}
\begin{proof}
For the upper estimates, suppose there are endpoint ratios \(\alpha_j\) and points \(z_j\in\supp\mu_{\alpha_j,P}\) staying a fixed distance from the indicated endpoint set. Fixed-ratio weak convergence supplies increasing \(n_j\) and zeros of \(R_{P;\lfloor\alpha_jn_j\rfloor,n_j}\) within \(1/j\) of \(z_j\), contrary to Theorem~\ref{thm:endpoint-zeros}.

At the lower endpoint, Proposition~\ref{prop:atoms} puts every root in the support for all sufficiently small \(\alpha\). It also gives
\[
 \mu_{\alpha,P}\Longrightarrow\rootmu_P\qquad(\alpha\downarrow0),
\]
because for \(0<\alpha<1\) the nonatomic mass is \(\alpha(s-1)/(d-\alpha)\). Let \(c\in Z(P')\setminus Z(P)\). If supports for a sequence \(\alpha_j\downarrow0\) avoided a fixed disk about \(c\), then \(C_{\alpha_j,P}\) would be holomorphic and zero-free there by Proposition~\ref{prop:selected-C}. On a smaller disk, weak convergence and separation from the supports give uniform convergence to \(\rootC_P=P'/(dP)\). The limit is not identically zero but vanishes at \(c\), contradicting Hurwitz's theorem. Thus every such critical point of \(P\) is approached by the supports.

At the upper endpoint, the supports are nonempty and all their points approach the centroid. This is the full Hausdorff limit. The definition of \(\Sh_P\) now gives \eqref{eq:distinguished}.
\end{proof}

\begin{remark}
The lower endpoint is a limit of supports, not the support of the weak limit. The support pieces approaching critical points of \(P\) carry vanishing mass as \(\alpha\downarrow0\).
\end{remark}

\section{The full shadow and its limiting core; Proof of Theorems~\ref{thm:intro-sets} and \ref{thm:p1-perfect-summary} }\label{sec:all-orders}
We now combine the interior and the endpoint limits. Averaging gives the shadow measure, while complete zero convergence determines which points persist at arbitrarily large power levels. The remaining assertions about connectedness, transients, and the full closure then follow from compactness and planar topology.

\subsection{Averaging over the derivative order}
\begin{theorem}[The pooled limit]\label{thm:averaging}
The pooled measures converge weakly to \(\rho_P\), with \(\supp\rho_P=\Sh_P\). More generally, \(P_n\to P\) in fixed degree implies \(\overline\mu_{P_n,n}^{\mathrm{root}}\Longrightarrow\rho_P\). For each continuous test function, convergence to the corresponding \(\rho_P\) is uniform on compact coefficient families. Their limiting logarithmic potential, the \emph{shadow potential}, is
\begin{equation}\label{eq:shadow-potential}
 \ShadowPot_P(z):=\frac2{d^2}\int_0^d\Blev_P(z,\alpha)\dd\alpha-1,
\end{equation}
and \(\rho_P=(2\pi)^{-1}\Delta\ShadowPot_P\). The integral is interpreted in \(L^1_{\mathrm{loc}}\) and may be \(-\infty\) at an atom. Moreover,
\begin{equation}\label{eq:atoms-intro}
 \rho_P^{\mathrm{at}}=\sum_{j=1}^{s}p_j^2\delta_{a_j}.
\end{equation}
The nonatomic mass is \(1-\sum_jp_j^2\), and
\begin{equation}\label{eq:lower-Hausdorff}
 \sup_{z\in\Sh_P}\dist(z,\mathcal Z_{P,n})\longrightarrow0.
\end{equation}
\end{theorem}
\begin{proof}
Fix a continuous test function \(\varphi\) on \(\Hull_P\). The bounded step functions \(\alpha\mapsto\int\varphi\dd\mu_{P;\lfloor n\alpha\rfloor,n}^{\mathrm{root}}\) converge pointwise on \((0,d)\) by Theorem~\ref{thm:fixed}. Multiply by \(2(d-\alpha)/d^2\) and apply dominated convergence. The resulting integral differs from the pooled sum by \(O(n^{-1})\|\varphi\|_\infty\): on each interval of length \(1/n\) the weight varies by \(O(n^{-1})\), and \(\mathcal N_{d,n}=d^2n^2/2+O(n)\). This proves the measure limit without a uniform saddle expansion in \(\alpha\). For \(P_n\to P\), the same argument works on a common compact root enclosure: the varying- polynomial statement of Theorem~\ref{thm:fixed} gives pointwise convergence of the step functions at every interior ratio. Joint weak continuity from Lemma~\ref{lem:alpha-cont}, followed by dominated convergence in \eqref{eq:rho-intro}, gives continuity of \(P\mapsto\rho_P\). The sequential assertion and compactness then give the stated uniformity.

By definition, the support of the average is contained in \(\Sh_P\). Conversely, a nonnegative continuous function with positive integral against one \(\mu_{\alpha_0,P}\) has positive integral on an interval of ratios, by Lemma~\ref{lem:alpha-cont}. The averaging weight is positive there. Hence \(\supp\rho_P=\Sh_P\).

Integrate the fixed-ratio potential formula in local \(L^1\), using Lemma~\ref{lem:pot-compact} to justify Fubini. Since
\[
 \frac2{d^2}\int_0^d \mathfrak h_d(\alpha)\dd\alpha
 =2\int_0^1[-t\log t-(1-t)\log(1-t)]\dd t=1,
\]
we obtain \eqref{eq:shadow-potential}. Integrating \eqref{eq:fixed-atoms} gives mass \(r_j^2/d^2\) at \(a_j\); Tonelli's theorem applied to each singleton excludes all other atoms.

Every ball centered on \(\supp\rho_P\) has positive limiting mass. Weak convergence therefore forces a pooled zero into a smaller such ball for all large \(n\). A finite covering of the compact support gives \eqref{eq:lower-Hausdorff}.
\end{proof}

\begin{proposition}[Covariance and power invariance]\label{prop:covariance}
For \(\lambda\in\C\setminus\{0\}\) and \(\beta\in\C\), set \(P_{\lambda,\beta}(z)=\lambda^{-d}P(\lambda z+\beta)\) and \(A_{\lambda,\beta}(z)=\lambda z+\beta\). Then
\[
 (A_{\lambda,\beta})_*\rho_{P_{\lambda,\beta}}=\rho_P,
 \qquad A_{\lambda,\beta}(\Sh_{P_{\lambda,\beta}})=\Sh_P,
\]
and
\[
 A_{\lambda,\beta}(\AllSh_{P_{\lambda,\beta}})=\AllSh_P.
\]
For every integer \(k\ge1\),
\[
 \rho_{P^k}=\rho_P,\qquad \Sh_{P^k}=\Sh_P.
\]
\end{proposition}
\begin{proof}
A change of variable in the landscape gives
\[
 \Blev_{P_{\lambda,\beta}}(z,\alpha)=\Blev_P(\lambda z+\beta,\alpha)
 -(d-\alpha)\log|\lambda|.
\]
The integrated potential changes by \(-\log|\lambda|\), proving covariance of the measure and its support. The chain rule gives covariance at each finite level and hence for the full shadow. Likewise, \(\Blev_{P^k}(z,\alpha)=k\Blev_P(z,\alpha/k)\); changing variables in \eqref{eq:shadow-potential} proves power invariance.
\end{proof}

\subsection{Complete zero convergence}
\label{sec:full}
The interior and endpoint estimates now cover every derivative order. Throughout this subsection, polynomial convergence means coefficient convergence of sequences of polynomials in fixed degree.

\begin{theorem}[Hausdorff convergence over all orders]\label{thm:no-outliers}
If \(P_n\to P\) are monic of degree \(d\ge2\), then
\begin{equation}\label{eq:stable-Hausdorff}
 d_{\mathrm H}(\mathcal Z_{P_n,n},\Sh_P)\longrightarrow0.
\end{equation}
For fixed \(P\), in particular,
\begin{equation}\label{eq:complete-limit}
 \mathcal Z_{P,n}\longrightarrow\Sh_P,
 \qquad
 \bigcap_{M\ge1}\overline{\bigcup_{n\ge M}\mathcal Z_{P,n}}=\Sh_P.
\end{equation}
\end{theorem}
\begin{proof}
The root hulls lie in one compact set. Suppose \(z_n\in\mathcal Z_{P_n,n}\) and \(z_n\to z\). Choose an order \(m_n\) realizing each zero and pass to a subsequence with \(m_n/n\to\alpha\in[0,d]\).

If \(\alpha=0\) or \(d\), the proof of Theorem~\ref{thm:endpoint-zeros} applies to \(P_n\). The coefficient families and their first coefficients converge locally uniformly, and the lower bounds of Lemma~\ref{lem:small-coeff} remain uniform. Thus \(z\in Z(P)\cup Z(P')\) or \(z=\centroid_P\), respectively; both lie in \(\Sh_P\) by Theorem~\ref{thm:endpoint-support}.

Suppose \(0<\alpha<d\) and \(z\notin\Sh_P\). A disk about \(z\) disjoint from \(\Sh_P\) contains no root, and the fixed-ratio potential is harmonic there. Theorem~\ref{thm:fixed} makes the barrier harmonic as well. Lemma~\ref{lem:harmonic-gate} gives a one-saddle separating contour persisting near \((P,z,\alpha)\), and Proposition~\ref{prop:uniform-saddle} excludes zeros near \(z\) at \((P_n,m_n/n)\) for all large \(n\). This contradicts \(z_n\to z\) and proves the upper Hausdorff estimate.

For the lower estimate, the varying-polynomial statement of Theorem~\ref{thm:fixed} and the averaging proof of Theorem~\ref{thm:averaging} give \(\overline\mu_{P_n,n}^{\mathrm{root}}\Longrightarrow\rho_P\). Indeed, the same bounded step functions converge pointwise at every interior ratio, so dominated convergence applies. Every disk centered on \(\supp\rho_P=\Sh_P\) therefore meets \(\mathcal Z_{P_n,n}\) for all large \(n\). A finite cover gives the lower estimate. The intersection identity in \eqref{eq:complete-limit} follows from Hausdorff convergence.
\end{proof}

\begin{corollary}[Hausdorff continuity and uniform convergence]\label{cor:set-stability}
In each fixed degree, both \(P\mapsto\Sh_P\) and \(P\mapsto\AllSh_P\) are Hausdorff-continuous, including at root collisions. For every compact set \(\mathscr P\) of monic degree-\(d\) polynomials,
\[
 \sup_{P\in\mathscr P}d_{\mathrm H}(\mathcal Z_{P,n},\Sh_P)\longrightarrow0.
\]
\end{corollary}
\begin{proof}
For \(P_j\to P\), choose increasing \(n_j\) with \(d_{\mathrm H}(\mathcal Z_{P_j,n_j},\Sh_{P_j})<1/j\). Theorem~\ref{thm:no-outliers} gives \(\mathcal Z_{P_j,n_j}\to\Sh_P\), hence \(\Sh_{P_j}\to\Sh_P\). Compactness and the same theorem give uniformity by contradiction.

For the full shadow, approximate a finite net in \(\AllSh_P\) by finite-level zeros. Continuity of the roots of each nonconstant derivative gives the lower estimate. For the upper estimate, approximate a proposed sequence of offending points in \(\AllSh_{P_j}\) by zeros at power levels \(n_j\). If the power levels are bounded, pass to fixed \(n,m\) and use root continuity. Otherwise, pass to increasing power levels and apply Theorem~\ref{thm:no-outliers}. In either case the limit belongs to \(\AllSh_P\). The common compact enclosure comes from Gauss--Lucas.
\end{proof}

\subsection{Connectedness of the core}

The following elementary fact applies to arbitrary compact sets.

\begin{lemma}[A continuous compact-set family with a singleton fiber]
\label{lem:connected-union}
Let \(I\) be a compact interval and let \(t\mapsto E_t\) be a Hausdorff-continuous family of nonempty compact subsets of a metric space with metric \(\delta\). If \(E_{t_*}\) is a singleton for some \(t_*\in I\), then \(\bigcup_{t\in I}E_t\) is compact and connected.
\end{lemma}
\begin{proof}
The union \(E\) is compact: from any sequence \(x_j\in E_{t_j}\) choose \(t_j\to t\in I\); Hausdorff continuity supplies \(y_j\in E_t\) with \(\delta(x_j,y_j)\to0\), and compactness of \(E_t\) supplies a convergent subsequence. If \(E\) were disconnected, write it as the disjoint union of two nonempty compact sets \(A\) and \(B\). They have positive distance. The sets
\[
 I_A=\{t:E_t\cap A\ne\varnothing\},\qquad
 I_B=\{t:E_t\cap B\ne\varnothing\}
\]
are closed by compactness and open by Hausdorff continuity and the positive distance between \(A\) and \(B\). Both are nonempty, so connectedness of \(I\) forces \(I_A=I_B=I\). This contradicts the singleton fiber at \(t_*\). Thus \(E\) is connected.
\end{proof}

\begin{theorem}[Connectedness and distinguished points of the core]
\label{thm:connected}
For every nonzero polynomial \(P\), the limiting core \(\Sh_P\) is connected. For a nonconstant polynomial it is compact and
\begin{equation}\label{eq:canonical-continuum-chain}
 Z(P)\cup Z(P')\cup\{\centroid_P\}
 \subseteq\Sh_P\subseteq\AllSh_P\subseteq\Hull_P.
\end{equation}
If \(P\) has at least two distinct roots, the core is perfect.
\end{theorem}
\begin{proof}
For degree \(d\ge2\), define the endpoint-completed support family by
\[
 E_P(t)=
 \begin{cases}
 Z(P)\cup Z(P'),&t=0,\\
 \supp\mu_{dt,P},&0<t<1,\\
 \{\centroid_P\},&t=1.
 \end{cases}
\]
Theorem~\ref{thm:endpoint-support} and the joint continuity in Theorem~\ref{thm:fixed-Hausdorff} make this a Hausdorff-continuous family of nonempty compact sets. Its union is compact, and it is exactly \(\Sh_P\) by the definition of the core and the endpoint limits. Lemma~\ref{lem:connected-union} proves connectedness. Equation~\eqref{eq:distinguished} gives the first inclusion in \eqref{eq:canonical-continuum-chain}; the middle inclusion follows from \eqref{eq:complete-limit}, and the last is Gauss--Lucas. A connected closed set containing two distinct points has no isolated point, so it is perfect. A nonzero constant has empty full shadow and empty core, while a polynomial with one distinct root has singleton full shadow and singleton core; these cases are immediate.
\end{proof}

Thus the core is a canonical connected refinement of the Gauss--Lucas hull. The singleton endpoint is essential to the connectedness argument.

Figure~\ref{fig:core-boundary}(a) shows the core as the union of the continuous family used in the proof. For small \(\alpha\) the support consists of the roots, which are atoms, and of short arcs near the critical points. The arcs grow, absorb the atoms at the simple roots when \(\alpha=1\) (Lemma~\ref{lem:threshold-root}), and then shrink to the centroid as \(\alpha\to d\) from below.

\subsection{Two mechanisms forming the boundary of the core}
\label{rem:boundary-mechanisms}
This subsection gives some insight about the boundary of a core which will be further developed and extended in \cite{Shadow 2}.

The supports \(\supp\mu_{\alpha,P}\) produce the boundary of \(\Sh_P\) in two different ways. In the examples, and typically, \(\supp\mu_{\alpha,P}\) is a finite union of analytic arcs. Along an arc two continued simple saddles \(w_A(z,\alpha)\) and \(w_B(z,\alpha)\) are the connecting ones and have equal critical heights \(\psi_A=\psi_B\), where \(\psi_j=\Ree\Phi_{P,z,\alpha}(w_j)\). Arcs may meet at junctions, where three saddles have the same height, and their free endpoints are critical values of \(F_{P,\alpha}\). At a critical value \(z=F_{P,\alpha}(w)\) one has \(\alpha h_P'(w)=1\), so, as \(\alpha\) varies, the endpoints move on the curve
\[
 \Bigl\{\,w-\frac{h_P(w)}{h_P'(w)}:\ P(w)P'(w)\ne0,\
 \frac1{h_P'(w)}\in(0,d)\Bigr\},
\]
that is, the image under the Newton map \(w\mapsto w-h_P(w)/h_P'(w)\) of \(h_P=P/P'\) of the part of the real-algebraic curve \(\{\operatorname{Im}h_P'=0\}\) on which \(h_P'>1/d\). Parts of \(\partial\Sh_P\) are traced by these moving endpoints.

A boundary point can also lie in the relative interior of an arc. Suppose that \(z_0\in\partial\Sh_P\) lies on \(\supp\mu_{\alpha_0,P}\) and that the connecting saddles at \((z_0,\alpha_0)\) are exactly two simple saddles \(w_A\ne w_B\). Put \(E=\psi_A-\psi_B\). Since \(w_j\) is a stationary point of the phase,
\[
 \partial_\alpha\psi_j=-\log|w_j-z|,\qquad
 \partial_\alpha E=\log|w_B-z|-\log|w_A-z|.
\]
By Lemma~\ref{lem:multiple-gates}, near \((z_0,\alpha_0)\) the barrier is \(\max(\psi_A,\psi_B)\), whose gradients in \(z\) are distinct because \(w_A\ne w_B\). Hence, near \(z_0\), the support of \(\mu_{\alpha,P}\) is the curve \(\{E(\cdot,\alpha)=0\}\) for every \(\alpha\) near \(\alpha_0\). If \(\partial_\alpha E(z_0,\alpha_0)\ne0\), the implicit-function theorem gives, for each \(z\) near \(z_0\), a ratio \(\alpha(z)\) with \(E(z,\alpha(z))=0\). Then \(z\in\supp\mu_{\alpha(z),P}\subseteq\Sh_P\), contradicting \(z_0\in\partial\Sh_P\). Therefore \(|w_A-z_0|=|w_B-z_0|\). Since the heights agree and \(w_j-z_0=\alpha_0h_P(w_j)\), this gives
\[
 |P(w_A)|=|P(w_B)|,\qquad |P'(w_A)|=|P'(w_B)|.
\]
Geometrically, such a boundary point lies on an envelope (a fold) of the family \(\{\supp\mu_{\alpha,P}\}\) rather than on the path of an endpoint. On an axis of symmetry of the root configuration the condition holds identically, and the core may contain one-dimensional pieces there, like the segment of the negative real axis for \(\Pi_2\).

Both mechanisms occur; see Figure~\ref{fig:core-boundary}. For the quartic in part~(a), the whole boundary is traced by endpoints, while other endpoint paths run through the interior, where they are covered by the supports of other ratios. For \(\Pi_2\) (part~(b) and Figure~\ref{fig:pi2-transients}), the two boundary arcs through the corner at the origin are folds. They consist of interior points of the supports with \(\alpha\) close to \(2\), and no endpoint path passes there. Folds also appear next to a crossing of two endpoint paths (part~(c)). A systematic description of \(\partial\Sh_P\) will be given in \cite{Shadow caustics}. There it is shown that for cubic polynomials the boundary is always contained in the planar closure of the ordinary critical values of the maps \(F_{P,\alpha}\), \(0<\alpha<3\); this is the compactified essential caustic, including its endpoint-ratio limits. In contrast, a quartic can have boundary arcs that avoid all these critical values. It is also shown that, for every \(0<a<4\), the two boundary arcs of \(\Sh_{\Pi_a}\) at the origin are folds. A separate quartic construction shows that even a simple critical point of \(P\), and the initial caustic germs issuing from it, can lie in \(\operatorname{int}\Sh_P\); this phenomenon persists under coefficient perturbations. Thus neither being a critical point of \(P\) nor belonging to a caustic endpoint path is sufficient for boundary membership.

\begin{figure}[htbp]
\centering
\begin{minipage}[b]{0.555\textwidth}\centering
 \includegraphics[width=\linewidth]{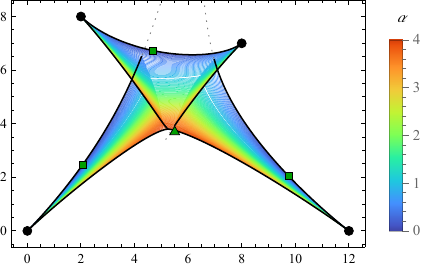}\\[1pt]
 {\small (a)}
\end{minipage}\hfill
\begin{minipage}[b]{0.425\textwidth}\centering
 \includegraphics[width=\linewidth]{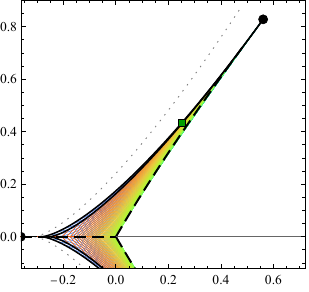}\\[1pt]
 {\small (b)}
\end{minipage}\\[6pt]
\begin{minipage}[b]{0.17\textwidth}\centering
 \raisebox{0.1661\linewidth}{\includegraphics[width=\linewidth]{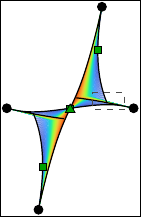}}\\[1pt]
{\small\phantom{(c)}}
\end{minipage}\quad
\begin{minipage}[b]{0.62\textwidth}\centering
 \includegraphics[width=\linewidth]{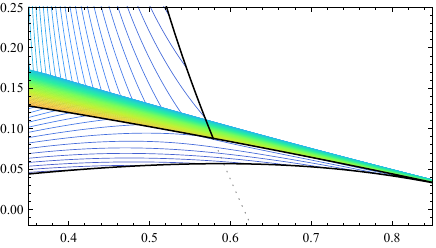}\\[1pt]
 {\small (c)}
\end{minipage}
\caption{The limiting core as the union of the supports \(\supp\mu_{\alpha,P}\) (thin curves, colored by \(\alpha\)). Black curves: the paths of the endpoints of the supports, which are critical values of \(F_{P,\alpha}\); dotted gray: critical values of \(F_{P,\alpha}\) that are not endpoints. (a) \(P(z)=z(z-12)(z-2-8i)(z-8-7i)\), \(\alpha=k/30\). The whole boundary is traced by endpoints. (b) \(\Pi_2\) near the corner at the origin, \(\alpha=k/24\); dashed: \(\supp\mu_{2,\Pi_2}\). The edge of the upper arm facing the positive real axis is a fold formed by supports with \(\alpha\) near \(2\); the opposite edge is traced by endpoints. (c) \(P(z)=(z^2-1)(z^2-c^2)\) with \(c=\frac12+\frac85i\), near a point where two endpoint paths cross (the box in the small overview). To the right of the crossing the core extends slightly beyond both paths, and the boundary there is again a fold.}
\label{fig:core-boundary}
\end{figure}

\subsection{Transient zeros}
\begin{theorem}[Decomposition of the full shadow]\label{thm:literal}
Let \(T_P=\AllSh_P\setminus\Sh_P\). Then
\begin{equation}\label{eq:literal-decomp}
 \AllSh_P=\Sh_P\cup T_P,
 \qquad
 T_P\subset\bigcup_{n\ge1}\mathcal Z_{P,n}.
\end{equation}
Every point of \(T_P\) is isolated in \(\AllSh_P\) and occurs at only finitely many power levels. The set \(T_P\) is countable and locally finite in \(\C\setminus\Sh_P\): each compact subset of that open set meets \(T_P\) in finitely many points. Consequently,
\begin{equation}\label{eq:literal-interior}
 \operatorname{int}\AllSh_P=\operatorname{int}\Sh_P,
 \qquad \Area(\AllSh_P)=\Area(\Sh_P),
 \qquad \partial\AllSh_P=\partial\Sh_P\cup T_P.
\end{equation}
\end{theorem}
\begin{proof}
A point in the closure of the finite-level union is either a limit along unbounded power levels, and hence lies in \(\Sh_P\) by \eqref{eq:complete-limit}, or is a zero at a bounded power level. If a compact set avoids \(\Sh_P\), the upper Hausdorff estimate excludes all sufficiently large power levels from a small neighborhood of it. Only finitely many finite-level zeros remain there. This proves local finiteness, isolation, and finite occurrence, as well as countability. Adding isolated points outside a closed set changes neither its interior nor its area, and gives the boundary identity in \eqref{eq:literal-interior}.
\end{proof}

A \emph{hole} of a closed planar set is a bounded connected component of its complement.

\begin{corollary}[The full shadow retains the core's holes]
\label{cor:full-holes}
For every component \(\Omega\) of \(\C\setminus\Sh_P\), the set \(\Omega\setminus T_P\) is a component of \(\C\setminus\AllSh_P\). This gives a bijection of complementary components preserving boundedness, area, and closure. In particular, the full shadow and the core have exactly the same number of holes.
\end{corollary}
\begin{proof}
The set \(T_P\) is locally finite in \(\C\setminus\Sh_P\) by Theorem~\ref{thm:literal}. The open connected set \(\Omega\) is polygonally connected. A polygonal path joining two points of \(\Omega\setminus T_P\) has a compact neighborhood contained in \(\Omega\), which meets \(T_P\) in finitely many points. Small detours inside that neighborhood avoid those points. Thus \(\Omega\setminus T_P\) is path-connected. It is nonempty, open, and dense in \(\Omega\), and components coming from distinct \(\Omega\) cannot join in the smaller complement. These are therefore exactly the complementary components of \(\AllSh_P\). Density preserves their closures and boundedness. Countability of \(T_P\) preserves their areas.
\end{proof}

\begin{remark} Note that in all illustrations in this text the core is simply connected. But as we will show in \cite{Shadow 2}, this is not always the case and therefore Corollary~\ref{cor:full-holes} is a non-trivial statement.
\end{remark}

\begin{proposition}[Passing to powers removes the transients]\label{prop:power-transients}
For every \(k\ge1\),
\[
 \Sh_{P^k}=\Sh_P\subseteq\AllSh_{P^k}\subseteq\AllSh_P,
 \qquad
 \bigcap_{k\ge1}\AllSh_{P^k}=\Sh_P,
 \qquad
 d_{\mathrm H}(\AllSh_{P^k},\Sh_P)\longrightarrow0\quad(k\to\infty).
\]
Moreover, if \(k\mid j\), then \(\AllSh_{P^j}\subseteq\AllSh_{P^k}\).
\end{proposition}
\begin{proof}
The descendants of powers of \(P^k\) use exactly the power levels of \(P\) divisible by \(k\); this also proves the divisibility inclusion. Proposition~\ref{prop:covariance} gives power invariance of the limiting measure and core. A transient occurs at only finitely many original power levels and is absent once \(k\) exceeds them all. Every power level used for \(P^k\) is at least \(k\), and the upper Hausdorff estimate for \(P\) is uniform over that tail. Taking closures proves the Hausdorff limit and the intersection formula.
\end{proof}

\subsection{Finite truncations and the full shadow}

The limit of a single power level describes the core. Taking all power levels up to a cutoff instead describes the full shadow, uniformly also when the coefficients vary.

\begin{theorem}[Uniform convergence of finite prefixes to the full shadow]
\label{thm:cumulative-shadow}
For \(M\ge1\) put
\[
 \varepsilon_M(P):=\sup_{n\ge M}d_{\mathrm H}(\mathcal Z_{P,n},\Sh_P).
\]
Then \(\varepsilon_M(P)\to0\) and
\begin{equation}\label{eq:cumulative-bound}
 d_{\mathrm H}(\mathcal F_{P,M},\AllSh_P)\le2\varepsilon_M(P).
\end{equation}
Both convergences are uniform on compact sets of monic polynomials of a fixed degree. Thus \(P_n\to P\) in that coefficient space implies
\[
 \mathcal F_{P_n,n}\longrightarrow\AllSh_P,
 \qquad \mathcal Z_{P_n,n}\longrightarrow\Sh_P.
\]
These are in general different limits.
\end{theorem}
\begin{proof}
Theorem~\ref{thm:no-outliers} and Corollary~\ref{cor:set-stability} give \(\varepsilon_M\to0\), uniformly on each compact coefficient family. A zero from a power level at most \(M\) already lies in \(\mathcal F_{P,M}\). A zero from a later power level is within \(\varepsilon_M\) of \(\Sh_P\), and every point of \(\Sh_P\) is within \(\varepsilon_M\) of \(\mathcal Z_{P,M}\subset\mathcal F_{P,M}\). Thus every zero, and then every point of their closure, is within \(2\varepsilon_M\) of the prefix. The reverse distance is zero since \(\mathcal F_{P,M}\subset\AllSh_P\). This proves \eqref{eq:cumulative-bound} and its uniform version. Combine it with the continuity of the two compact-set maps to obtain the varying-polynomial statements.
\end{proof}

\begin{proof}[Proof of Theorem~\ref{thm:intro-sets}]
The pooled measure convergence, the identity \(\supp\rho_P=\Sh_P\), continuity of \(P\mapsto\rho_P\), and the uniform measure statement are Theorem~\ref{thm:averaging}. Convergence of the complete zero sets and the tail formula \eqref{eq:tail-core-intro} are Theorem~\ref{thm:no-outliers}. Hausdorff continuity of \(P\mapsto\Sh_P\) and \(P\mapsto\AllSh_P\), together with uniform complete-zero convergence, are Corollary~\ref{cor:set-stability}. Finally, convergence of the finite prefixes to the full shadow, including the varying-polynomial and uniform forms, is Theorem~\ref{thm:cumulative-shadow}.
\end{proof}

\subsection{Accumulation points and the perfect kernel}
\begin{corollary}[Accumulation points of all finite zeros]\label{cor:realization}
If \(P\) has at least two distinct roots, then
\[
 \acc\!\left(\bigcup_{n\ge1}\mathcal Z_{P,n}\right)=\acc(\AllSh_P)=\Sh_P.
\]
For a power of a linear polynomial, the full shadow is a singleton and both accumulation sets are empty.
\end{corollary}
\begin{proof}
Connectedness and the presence of two roots imply that \(\Sh_P\) has no isolated points. Given a point of \(\Sh_P\) and any neighborhood of it, choose another support point in that neighborhood and a smaller disk around the latter avoiding the original point. The disk has positive limiting mass and meets \(\mathcal Z_{P,n}\) for every sufficiently large \(n\). Thus every point of \(\Sh_P\) is an accumulation point of the finite-level union. Conversely, a sequence of distinct finite-level zeros uses unbounded power levels, so Theorem~\ref{thm:no-outliers} places its limit in \(\Sh_P\). Theorem~\ref{thm:literal} excludes further accumulation points of the full closure. The one-root case follows by direct differentiation.
\end{proof}

\begin{corollary}[The perfect kernel of the full shadow]
\label{cor:perfect-kernel}
If \(P\) has at least two distinct roots, then \(\Sh_P\) is the unique largest perfect subset of \(\AllSh_P\). Thus a single accumulation-set operation recovers the core from the full shadow:
\[
 \acc(\AllSh_P)=\Sh_P=\acc(\Sh_P).
\]
\end{corollary}
\begin{proof}
The set \(\Sh_P\) is closed, connected, and contains two points, so it is perfect. Any perfect subset \(E\subseteq\AllSh_P\) satisfies \(E=\acc (E)\subseteq\acc(\AllSh_P)=\Sh_P\) by Corollary~\ref{cor:realization}. This proves maximality and uniqueness.
\end{proof}

\begin{proof}[Proof of Theorem~\ref{thm:p1-perfect-summary}]
Compactness, connectedness, the distinguished points, and perfection of the core are Theorem~\ref{thm:connected}. The decomposition of the full shadow, finite occurrence and local finiteness of the transients, and the interior and area statements are Theorem~\ref{thm:literal}. The intersection formula, Hausdorff convergence under passage to powers, and nesting along divisibility are Proposition~\ref{prop:power-transients}. The accumulation-set and perfect-kernel assertions are Corollaries~\ref{cor:realization} and~\ref{cor:perfect-kernel}. The claimed failure of nesting is Proposition~\ref{prop:non-nesting}.
\end{proof}

\section{Quartics with infinitely many transient zeros; proof of Theorem~\ref{thm:intro-infinite}}
\label{sec:infinite-transients}
We finish by showing that the transient set \(T_P\) of the full shadow need not be finite. An exact modulus inequality identifies the relevant contours. A three-point asymptotic expansion then produces positive zeros outside the core at every sufficiently large power level. For \(0<a<4\), put
\begin{equation}\label{eq:infinite-quartic}
 \Pi_a(z):=z^4+az^3-\beta_a z^2+az+1,\qquad
 \beta_a:=a-\frac{a^2}{8},\qquad M_a:=\frac{(a+4)^2}{8},
\end{equation}
and define
\begin{equation}\label{eq:infinite-parameters}
 \theta_a:=\arccos(-a/4)\in(\pi/2,\pi),\qquad
 \omega_a:=\frac{\sqrt{16-a^2}}2,\qquad
 \chi_a:=\frac4{\sqrt{4-a}}.
\end{equation}
The family allows multiple roots: at \(a=12-8\sqrt2\) there is a double root at \(-1\). No member is excluded on that account.

\subsection{An exact modulus certificate}
\begin{lemma}[The three equality points]\label{lem:infinite-modulus}
For every \(r>0\),
\[
 \Pi_a(r)>0,\qquad \Pi_a'(r)>0,\qquad
 \frac{\dd}{\dd r}\frac{r\Pi_a'(r)}{\Pi_a(r)}>0,
 \qquad |\Pi_a(re^{i\theta})|\le\Pi_a(r).
\]
Apart from \(\theta\in2\pi\mathbb Z\), equality in the last inequality holds exactly when \(r=1\) and \(\theta\equiv\pm\theta_a\pmod{2\pi}\).
\end{lemma}
\begin{proof}
Set \(y=\cos\theta\) and \(\upsilon=r+r^{-1}-2\ge0\). Direct expansion gives
\begin{equation}\label{eq:infinite-modulus-factor}
 \Pi_a(r)^2-|\Pi_a(re^{i\theta})|^2
   =2(1-y)r^4\,\mathcal E_a(\upsilon,y),
\end{equation}
where
\begin{align*}
 \mathcal E_a(\upsilon,y)={}&a\upsilon^3
 +\frac a4\bigl((a-8)y+a+16\bigr)\upsilon^2\\
 &+\frac a8\bigl(32y^2+(8a-32)y+a^2+16\bigr)\upsilon
 +\frac{(a+4y)^2(a+2y+2)}4.
\end{align*}
For \(-1\le y\le1\), all three coefficients of positive powers of \(\upsilon\) are strictly positive. For the linear coefficient, the quadratic in \(y\) has positive leading coefficient and discriminant \(-64(a+4)^2<0\). The constant term is nonnegative, and vanishes exactly at \(y=-a/4\). This proves the modulus inequality and its equality cases.

Writing \(u=r+r^{-1}\ge2\) gives
\[
 \Pi_a(r)=r^2\bigl(u^2+au-2-\beta_a\bigr)\ge M_a r^2>0.
\]
Also \(\Pi_a'(r)=4r^3+3ar^2-2\beta_a r+a>0\): the quadratic \(3ar^2-2\beta_a r+a\) has negative discriminant, since \(0<\beta_a<a\). Finally, with \(\mathfrak D_a(r)=r\Pi_a'(r)/\Pi_a(r)\), the coefficient of \(\theta^2\) in the left side of \eqref{eq:infinite-modulus-factor} is \(\Pi_a(r)^2r\mathfrak D_a'(r)\). On the right it is \(r^4\mathcal E_a(\upsilon,1)>0\). Thus \(\mathfrak D_a'(r)>0\).
\end{proof}

\begin{proposition}[A ray outside the limiting core and three connecting saddles]
\label{prop:infinite-ray}
For every \(0<a<4\),
\begin{equation}\label{eq:infinite-ray}
 \Sh_{\Pi_a}\cap(0,\infty)=\varnothing,
 \qquad 0\in\supp\mu_{2,\Pi_a}\subseteq\Sh_{\Pi_a}.
\end{equation}
At \((z,\alpha)=(0,2)\) the barrier is \(\log M_a\). Its connecting saddles are exactly \(1,e^{i\theta_a},e^{-i\theta_a}\), and all three are simple.
\end{proposition}
\begin{proof}
Fix \(x>0\). For \(r>0\) set \(u=x+r\) and \(\mathfrak D_{a,x}(r)=r\Pi_a'(x+r)/\Pi_a(x+r)\). If \(\lambda_a(u)=\log\Pi_a(u)\), then
\[
 \mathfrak D_{a,x}'(r)=\frac xu\lambda_a'(u)
       +\frac ru\bigl(\lambda_a'(u)+u\lambda_a''(u)\bigr)>0
\]
by Lemma~\ref{lem:infinite-modulus}. Its endpoint values are \(0\) and \(4\). For every \(0<\alpha<4\) there is therefore a unique \(r\) with \(\mathfrak D_{a,x}(r)=\alpha\). On the circle \(w=x+re^{i\theta}\),
\[
 |\Pi_a(w)|\le\Pi_a(|w|)\le\Pi_a(x+r),
\]
with equality only at \(\theta=0\); use strict increase of \(\Pi_a\) and \(|x+re^{i\theta}|<x+r\) otherwise. The angular maximum of the landscape is nondegenerate, with second derivative \(-r\mathfrak D_{a,x}'(r)<0\). The positive ray from \(x\) to infinity has its unique landscape minimum at the same point. The ray gives the lower min--max bound and the circle gives the matching upper bound. The circle separates \(x\) from infinity and is strictly below the common barrier value except at the stationary point. Hence the two strict superlevel components lie on opposite sides of the circle and their closures can meet the circle only at that point. It is therefore the unique connecting saddle, and the circle is a one-saddle separating contour. A circle passing through a root may be indented there, where the landscape is strictly below its maximum. The resulting contour persists for nearby observation points. Proposition~\ref{prop:uniform-saddle} therefore excludes \(x\) from every fixed-ratio support. The endpoint supports also stay away from \(x\): \(\Pi_a(x)\Pi_a'(x)>0\) and the centroid is \(-a/4<0\). By Hausdorff continuity, the support family completed at the two endpoints has compact union. This union is the core and misses \(x\), proving the first assertion of \eqref{eq:infinite-ray}.

At \(z=0\) the saddle equation for \(\alpha=2\) factors as
\[
 w\Pi_a'(w)-2\Pi_a(w)=(w^2-1)(2w^2+aw+2).
\]
On the unit circle,
\begin{equation}\label{eq:infinite-symbol}
 e^{-2i\theta}\Pi_a(e^{i\theta})
 =f_a(\theta):=2\cos(2\theta)+2a\cos\theta-\beta_a.
\end{equation}
Its largest absolute value is \(M_a\), attained only at \(0,\theta_a,-\theta_a\), with values \(M_a,-M_a,-M_a\), respectively. The fourth candidate, \(w=-1\), has strictly smaller height whenever it is a finite saddle. Roots on the circle can be avoided by small indentations strictly below \(\log M_a\), leaving the three equality neighborhoods unchanged. The positive radial path attains its minimum \(\log M_a\) at \(w=1\). The circle upper bound and radial lower bound therefore give the barrier \(\log M_a\).

The whole unit circle has landscape at most \(\log M_a\). At each of its three equality points the angular maximum is nondegenerate. The two local higher-level sectors lie on opposite sides of the circle. No strict superlevel component can cross that circle, and every such component contains a positive end. The inner and outer sectors therefore belong to the components of \(0\) and infinity, respectively. All three saddles are connecting. Define their normalized angular curvatures \(\varsigma_0,\varsigma_+,\varsigma_-\) by
\[
 \log|f_a(\theta_*+t)|
 =\log M_a-\frac12\varsigma_*t^2+O(t^3),
 \qquad \theta_*\in\{0,\theta_a,-\theta_a\}.
\]
A direct differentiation gives
\begin{equation}\label{eq:infinite-variances}
 \varsigma_0=\frac{16}{a+4},\qquad
 \varsigma_+=\varsigma_-=\frac{4(4-a)}{a+4}.
\end{equation}
All three curvatures are positive. Lemma~\ref{lem:multiple-gates} now expresses the nearby barrier as the maximum of the three continued critical heights. The values \(2\partial_z\psi_j=2/w_j\) are distinct, so that maximum is not harmonic at the origin. The potential formula gives \(0\in\supp\mu_{2,\Pi_a}\).
\end{proof}

Figure~\ref{fig:three-saddles} shows this landscape for \(a=2\). The three connecting saddles on the unit circle are the source of the interference in Theorem~\ref{thm:infinite-transients} below.

\begin{figure}[htbp]
\centering
\includegraphics[width=0.62\textwidth]{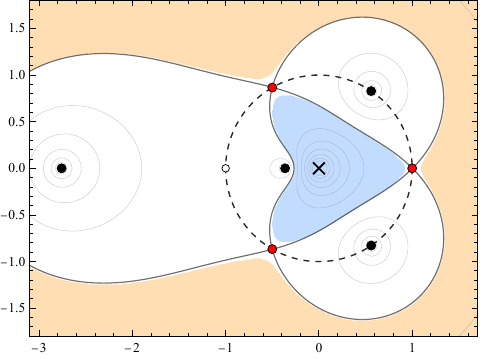}
\caption{The landscape \(L_{\Pi_2,0,2}(w)=\log|\Pi_2(w)|-2\log|w|\) of Proposition~\ref{prop:infinite-ray}. The unit circle (dashed) is a barrier contour: on it the landscape is at most \(\log M_2=\log\frac92\), with equality exactly at the three connecting saddles \(1\) and \(e^{\pm2\pi i/3}\) (red). The fourth saddle \(-1\) (white) is lower and is not connecting. Blue: the component of \(\{L_{\Pi_2,0,2}>\log M_2\}\) containing \(0\); orange: the component containing \(\infty\); black disks: the roots of \(\Pi_2\).}
\label{fig:three-saddles}
\end{figure}

\begin{proposition}[Failure of natural-order nesting]
\label{prop:non-nesting}
For \(P=\Pi_2\), there are points
\[
 \xi_2\in\AllSh_{P^2}\setminus\AllSh_{P^3},\qquad
 \xi_3\in\AllSh_{P^3}\setminus\AllSh_{P^2}.
\]
In particular, \(\AllSh_{P^2}\) and \(\AllSh_{P^3}\) are incomparable.
\end{proposition}
\begin{proof}
Let
\[
 P=\Pi_2=z^4+2z^3-\frac32z^2+2z+1,
 \qquad B_x(u):=P(x+u),
\]
and put
\begin{align*}
 g_2(x)&:=28x^3+42x^2+3x-1,\\
 g_3(x)&:=2464x^6+7392x^5+4200x^4-896x^3
          +1176x^2+588x-65.
\end{align*}
Direct differentiation gives
\begin{equation}\label{eq:nonnesting-derivatives}
 (P^2)^{(5)}(x)=240g_2(x),
 \qquad
 (P^3)^{(6)}(x)=270g_3(x).
\end{equation}
Let
\[
 I_2=\left[\frac{119}{1000},\frac3{25}\right],
 \qquad
 I_3=\left[\frac{93}{1000},\frac{47}{500}\right].
\]
Exact substitution gives opposite signs for \(g_j\) at the two endpoints of \(I_j\). Moreover, \(g_2'(x)>0\) on \(I_2\), while on \(I_3\)
\[
 g_3'(x)>
 588+2352\frac{93}{1000}-2688\left(\frac{47}{500}\right)^2>0.
\]
Thus there are unique points \(\xi_j\in I_j\) with \(g_j(\xi_j)=0\). In particular, \eqref{eq:nonnesting-derivatives} puts \(\xi_2\) at power level \(2\) and \(\xi_3\) at power level \(3\).

For completeness, the coefficient signs used below admit a short exact certificate. If \(c(x)\) is a coefficient polynomial of degree \(e_c\) on \(I=[x_-,x_+]\), write
\[
 c\bigl(x_-+(x_+-x_-)t\bigr)
 =\sum_{h=0}^{e_c}b_h\binom{e_c}{h}t^h(1-t)^{e_c-h}.
\]
The Bernstein basis is nonnegative on \(0\le t\le1\). Direct expansion of \(B_x(u)^N\) gives the following signs, in increasing powers of \(u\); the last column is the smallest absolute value of any Bernstein coefficient appearing in that row.
\[
\begin{array}{c|c|c|c}
 I & N & \text{coefficient signs} & \min|b_h|\\ \hline
 I_2&3&\text{all \(13\) positive}&1\\
 I_3&2&(+,+,+,+,+,-,+,+,+)&\dfrac{10113489}{15625000}\\
 I_3&4&\text{all \(17\) positive}&1\\
 I_3&6&\text{all \(25\) positive}&1
\end{array}
\]
For reproducibility, the following exact-arithmetic Mathematica code produces the sign lists and minima in the table; no floating-point evaluation is used.
\begingroup\small
\begin{verbatim}
p[x_] := x^4 + 2 x^3 - 3 x^2/2 + 2 x + 1;
bern[q_, {l_, r_}] := Module[{v, e},
  v = Expand[q /. x -> l + (r - l) t]; e = Exponent[v, t];
  Table[Sum[Coefficient[v, t, j] Binomial[h, j]/Binomial[e, j],
            {j, 0, h}], {h, 0, e}]];
row[N_, I_] := Module[{v},
  v = bern[#, I] & /@ CoefficientList[Expand[p[x + u]^N], u];
  {DeleteDuplicates[Sign[#]] & /@ v, Min[Abs[Flatten[v]]]}];
row[3, {119/1000, 3/25}]
row[2, {93/1000, 47/500}]
row[4, {93/1000, 47/500}]
row[6, {93/1000, 47/500}]
\end{verbatim}
\endgroup
This is a finite calculation over the rationals and proves the stated signs throughout the two intervals.

All coefficients of \(B_{\xi_2}(u)^3\) are therefore positive, and hence so are all coefficients of \(B_{\xi_2}(u)^{3j}\) for \(j\ge1\). Consequently \(\xi_2\) is not a descendant zero at any power level divisible by \(3\). For \(\xi_3\), no coefficient of \(B_{\xi_3}(u)^2\) vanishes, and all coefficients of its fourth and sixth powers are positive. Every even integer \(N\ge4\) has the form \(4j\) with \(j\ge1\) or \(6+4j\) with \(j\ge0\), so no coefficient of \(B_{\xi_3}(u)^N\) vanishes at any even power level either.

Both \(\xi_2\) and \(\xi_3\) are positive, so Proposition~\ref{prop:infinite-ray} puts them outside \(\Sh_P\). Using \(\Sh_{P^k}=\Sh_P\) from Proposition~\ref{prop:power-transients}, Theorem~\ref{thm:literal}, applied also to powers of \(P\), turns the preceding absence of finite-level zeros into absence from the corresponding full shadows. Hence
\[
 \xi_2\in\AllSh_{P^2}\setminus\AllSh_{P^3},
 \qquad
 \xi_3\in\AllSh_{P^3}\setminus\AllSh_{P^2}.
\]
\end{proof}

\subsection{Infinitely many transients for a fixed polynomial}
\begin{theorem}[Three-saddle interference and infinite transients]
\label{thm:infinite-transients}
Fix \(0<a<4\) and an integer \(h\). Uniformly for \(\sigma\) in compact subsets of \(\C\), as \(n\to\infty\),
\begin{align}
 &\frac{4\sqrt{2\pi n}}{\sqrt{a+4}\,M_a^n}
       \frac{(\Pi_a^n)^{(2n+h)}(\sigma/n)}{(2n+h)!}
 \label{eq:infinite-interference}\\
 &\hspace{1cm}=
 e^{2\sigma}+(-1)^n\chi_a e^{-a\sigma/2}
                   \cos(\omega_a \sigma+h\theta_a)+O(n^{-1}).\notag
\end{align}
The derivative order is in \([0,4n-1]\) for all sufficiently large \(n\).

Let \(\sigma_a\) be the unique solution in \((0,\pi/\sqrt{16-a^2})\) of
\begin{equation}\label{eq:infinite-scale}
 e^{(a+4)\sigma_a/2}=\chi_a\cos(\omega_a \sigma_a).
\end{equation}
For every sufficiently large odd \(n\), \((\Pi_a^n)^{(2n)}\) has a simple positive zero
\begin{equation}\label{eq:infinite-zero-scale}
 x_n=\frac{\sigma_a}{n}+O(n^{-2}).
\end{equation}
Every such zero is a transient. In particular \(T_{\Pi_a}\) is infinite. In fact every sufficiently large power level has a positive transient: for even \(n\) one can use order \(2n+h_a\), where
\[
 h_a=\begin{cases}2,&0<a\le2,\\1,&2<a<4.\end{cases}
\]
Its simple positive zero has the form \(\widetilde \sigma_a/n+O(n^{-2})\), where \(\widetilde \sigma_a>0\) is the first positive zero of \(e^{2\sigma}+\chi_a e^{-a\sigma/2}\cos(\omega_a \sigma+h_a\theta_a)\).
\end{theorem}
\begin{proof}
Cauchy's formula on the unit circle gives the exact integral
\begin{equation}\label{eq:infinite-cauchy}
 \frac{(\Pi_a^n)^{(2n+h)}(\sigma/n)}{(2n+h)!}
 =\frac1{2\pi}\int_{-\pi}^{\pi}
 f_a(\theta)^n e^{-ih\theta}
 \left(1-\frac{\sigma}{n}e^{-i\theta}\right)^{-2n-h-1}\dd\theta.
\end{equation}
For bounded complex \(\sigma\), the product of the last two factors, together with all \(\theta\)-derivatives needed below, converges uniformly to \(e^{-ih\theta}\exp(2\sigma e^{-i\theta})\), with error \(O(n^{-1})\). The only maxima of \(|f_a|\) are the three points in Proposition~\ref{prop:infinite-ray}. Near zero,
\[
 f_a(\theta)=M_a-(a+4)\theta^2+O(\theta^4),
\]
and near either \(\pm\theta_a\) it is \(-M_a\) plus a positive quadratic term with second derivative \((16-a^2)/2\). Laplace's method on three fixed disjoint neighborhoods therefore gives Gaussian weights \(\varsigma_0^{-1/2}\) and \(\varsigma_+^{-1/2}\), while the rest of the integral is exponentially smaller. Taylor expansion in each Gaussian coordinate makes the remainder after normalization \(O(n^{-1})\), uniformly for bounded \(\sigma\); the odd first-order terms integrate to zero. At the negative extrema \(f_a^n\) contributes \((-1)^n\). The two conjugate contributions combine to a cosine. Since \(2\varsigma_+^{-1/2}/\varsigma_0^{-1/2}=\chi_a\), this proves \eqref{eq:infinite-interference}. The estimate is locally uniform as an identity of holomorphic functions of \(\sigma\).

For odd \(n\) and \(h=0\), the limiting function is negative at zero. After multiplication by \(e^{a\sigma/2}\), its first positive zero is the intersection of the strictly increasing function \(e^{(a+4)\sigma/2}\) and the strictly decreasing positive function \(\chi_a\cos(\omega_a \sigma)\) on
\[
 (0,\pi/(2\omega_a))=(0,\pi/\sqrt{16-a^2}).
\]
It is unique and simple. Rouch\'e's theorem on a small circle around \(\sigma_a\) gives exactly one nearby zero, with error \(O(n^{-1})\) in the \(\sigma\) coordinate. Real coefficients and uniqueness make this zero real; it is positive and simple. This proves \eqref{eq:infinite-zero-scale}.

For even \(n\), the stated choice of \(h_a\) gives \(\chi_a\cos(h_a\theta_a)<-1\). If \(a\le2\), then \(2\theta_a\in(\pi,4\pi/3]\); before the cosine next reaches zero, \(-\cos(2\theta_a+\omega_a \sigma)\) decreases strictly while the exponential increases. If \(a>2\), then \(\theta_a\in(2\pi/3,\pi)\) and the ratio
\[
 \frac{e^{(a+4)\sigma/2}}{-\cos(\theta_a+\omega_a \sigma)}
\]
increases strictly until the denominator next vanishes. On the part before angle \(\pi\), its logarithmic derivative is \((a+4)/2+\omega_a\tan(\theta_a+\omega_a \sigma)\), whose initial value is \((a+4)(a-2)/a>0\); thereafter both terms are positive. The first positive zero is consequently simple in either case. The same Rouch\'e argument proves the assertion for even power levels.

Proposition~\ref{prop:infinite-ray} excludes every positive real point from the limiting core. Thus all these zeros belong to \(T_{\Pi_a}\). The zeros from odd power levels tend to zero through positive values, so there are infinitely many distinct ones. Their isolation follows from Theorem~\ref{thm:literal}.
\end{proof}

\begin{proof}[Proof of Theorem~\ref{thm:intro-infinite}]
Proposition~\ref{prop:infinite-ray} proves that the core of \(\Pi_a\) is disjoint from the positive real axis. Theorem~\ref{thm:infinite-transients} provides the stated three-saddle asymptotic, the simple positive zeros at all sufficiently large odd power levels with the displayed scale, their infinitude and transience, and the bounded order shift that supplies a positive transient at every sufficiently large even power level.
\end{proof}

A subset of \(\C\simeq\R^2\) is \emph{semialgebraic} if it is a finite Boolean combination of sets defined by real polynomial equations and inequalities in the two coordinates.

\begin{corollary}[Infinite topology of the full shadow]
\label{cor:infinite-topology}
For every \(0<a<4\), the full shadow of \(\Pi_a\) has infinitely many isolated connected components and is not semialgebraic. No finite list of power levels, even one chosen separately for this fixed polynomial, contains all its transients.
\end{corollary}
\begin{proof}
Theorem~\ref{thm:infinite-transients} and isolation give infinitely many singleton components. A semialgebraic set has only finitely many connected components \cite{BCR}. Each finite list of power levels contributes only finitely many zeros.
\end{proof}

\appendix
\section{First variation at degenerate connecting saddles}\label{app:variation}
The barrier-contour construction also gives a quantitative first-variation formula in all parameters. This result is not needed for the zero-set convergence proofs, but records the dependence at saddle collisions.

\begin{theorem}[First variation of the barrier]
\label{thm:barrier-first-variation}
Fix a monic degree-\(d\) polynomial \(P\) with \(s\) distinct roots, \(z_0\notin Z(P)\), and \(0<\alpha<d\). Let its connecting saddles be \(w_1,\ldots,w_k\), with phase orders \(q_1,\ldots,q_k\), and put \(\kappa_j=\alpha/(w_j-z_0)\). For a polynomial \(H\) of degree at most \(d-1\), a complex increment \(\zeta\), and a real increment \(\tau\), set \(\eta=\|H\|+|\zeta|+|\tau|\), using any fixed coefficient norm. Put \(q_*:=\max_j q_j\) and \(\theta_*:=1/(q_*-1)\ge1/s\). As \(\eta\to0\),
\begin{align}
 &\Blev_{P+H}(z_0+\zeta,\alpha+\tau)-\Blev_P(z_0,\alpha)\notag\\
 &\quad=\max_{1\le j\le k}
 \left\{\Ree\left(\frac{H(w_j)}{P(w_j)}+\kappa_j\zeta\right)
             -\tau\log|w_j-z_0|\right\}
                   +O(\eta^{1+\theta_*}).
 \label{eq:barrier-first-variation}
\end{align}
The constant depends on the base data, but not on the direction of the increments. In particular the error is always \(O(\eta^{1+1/s})\); if every connecting saddle is simple it is \(O(\eta^2)\). The barrier is locally Lipschitz in the monic coefficients, the nonroot observation point, and the interior derivative ratio. At fixed \(P,\alpha\), it is differentiable in \(z_0\) if and only if there is exactly one connecting saddle.
\end{theorem}
\begin{proof}
For each connecting saddle choose a locally ascending arc on either side, one in each of the two positive-end components. Join these arcs to fixed neighborhoods of the corresponding ends. After truncation, the result is a compact path whose unique minimum is the chosen saddle; the parts used for attachment have a strict height margin. Together with the barrier contour of Lemma~\ref{lem:arbitrary-contour}, these paths, which lie strictly above the barrier away from their saddles, provide the lower and upper variational bounds. On each of these fixed compact paths and on the contour, uniformly in \(w\),
\[
 L_{P+H,z_0+\zeta,\alpha+\tau}(w)
 =L_{P,z_0,\alpha}(w)+
 \Ree\left(\frac{H(w)}{P(w)}+\frac{\alpha\zeta}{w-z_0}\right)
 -\tau\log|w-z_0|+O(\eta^2).
\]
Here and below the error bounds are on fixed compact paths, so they are uniform over unit increment directions. In the local coordinate at \(w_j\), the unperturbed contour has height \(B-|v_j|^{q_j}\). Choose the ascending paths along centers of higher sectors, where the height is \(B+|v_j|^{q_j}\). The displayed linear perturbation changes from its value at \(w_j\) by at most \(C\eta|v_j|\). The elementary bound
\[
 \sup_{r\ge0}\{-r^q+C\eta r\}
       =O(\eta^{q/(q-1)}),\qquad q\ge2,
\]
and its negative give the upper contour bound and lower path bounds with errors \(O(\eta^{q_j/(q_j-1)}+\eta^2)\). Away from the saddle neighborhoods the fixed strict height margins make these estimates automatic for sufficiently small \(\eta\). The min--max formula bounds the perturbed barrier above by the contour maximum and below by the maximum of the path minima. Taking \(q_* =\max_j q_j\) proves \eqref{eq:barrier-first-variation} with the stated error. The phase-order budget \eqref{eq:phase-order-budget} gives \(q_*-1\le s\).

To justify the Lipschitz assertion also across saddle collisions, shrink a parameter neighborhood of the base data. By continuity of the barrier, every connecting saddle for data in its closure lies in one compact set away from roots, the observation point, and infinity: the phase height tends uniformly to \(-\infty\) near roots and to \(+\infty\) at the two positive ends. Thus every linear form on the right of \eqref{eq:barrier-first-variation}, now based at any point in that neighborhood, has norm bounded by one constant. On each straight parameter segment, the limsup and liminf of its one-sided difference quotients (the upper and lower Dini derivatives) are consequently bounded in absolute value by that constant times the segment speed. The elementary one-variable Dini-derivative mean-value inequality gives the Lipschitz estimate.

With only \(z_0\) varying, the first variation is \(\max_j\Ree(\kappa_j\zeta)\). Distinct connecting saddles give distinct \(\kappa_j\). A maximum of distinct real-linear forms is not linear. This proves both directions of the differentiability claim.
\end{proof}

\section{Equal critical heights without contour connectivity}\label{app:equal-height}
The condition that a saddle have barrier height is weaker than the condition that it connect the positive-end components.

\begin{example}[Equal height without contour contribution]
\label{ex:equal-height-not-connecting}
Take \(P(w)=w^2-1\), \(\alpha=1/2\), and \(z=3/(2\sqrt2)>1\). The two finite saddles are
\[
 w_\pm=\frac{3\pm\sqrt3}{3\sqrt2},\qquad
 \exp\bigl(2L_{P,z,1/2}(w_\pm)\bigr)=\frac{4\sqrt6}{9}.
\]
This common height is the barrier. Indeed, on the positive real ray from \(z\) to infinity the landscape has its unique minimum at \(w_+\). On the circle \(|w-z|=w_+-z\), both factors \(|w-1|\) and \(|w+1|\) are maximized uniquely at \(w=w_+\), so the circle has its unique maximum at that saddle. The two variational bounds coincide. The circle is a one-saddle separating contour through \(w_+\) and does not pass through \(w_-\). If \(w_-\) were connecting, an ascending path through it from one positive end to the other would have to cross that circle at height strictly greater than the barrier, which is impossible. Thus \(w_-\) is at barrier height but is not connecting. Also \(z\) is outside the root hull \([-1,1]\), and hence outside every limiting zero support.
\end{example}

\section{Differentiators and the essential spectrum}
\label{app:differentiator-spectrum}

The purpose of this appendix is to package the zero sets studied above into the spectrum of a single bounded operator. In this model, the full shadow becomes the spectrum, the limiting core becomes the Fredholm essential spectrum, and the transient set becomes the discrete spectrum. The finite-dimensional construction is classical \cite{Pereira}; the new input from the preceding results is the uniform resolvent control needed when these nonnormal matrices are assembled into an infinite direct sum.

For a polynomial \(H\) of degree \(r\ge1\), list its roots with multiplicity and let \(D_H\) be the resulting \(r\times r\) diagonal matrix. Put \(e_r=r^{-1/2}(1,\ldots,1)^T\) and \(Q_r=I-e_re_r^*\), the orthogonal projection onto \(e_r^\perp=\{v:e_r^*v=0\}\). For \(r\ge2\), set \(B_H=Q_rD_H|_{e_r^\perp}\), acting on \(e_r^\perp\). Here \(*\) denotes the Hermitian adjoint. Writing \(\widehat H\) for the monic normalization of \(H\), the classical differentiator identity is \(\det(zI-B_H)=\widehat H'(z)/r\); thus the eigenvalues of \(B_H\) are exactly the zeros of \(H'\), counted with multiplicity. To see the identity, extend \(B_H\) by zero on \(\C e_r\): its full-space matrix is \(Q_rD_HQ_r\), whose characteristic polynomial is \(z\widehat H'(z)/r\) by the rank-one determinant lemma. The zero summand contributes the factor \(z\), which cancels as a polynomial identity. For \(r\ge2\) and \(z\notin Z(H)\), define the diagonal resolvent and the root Cauchy transform by
\[
 \mathcal R_H(z)=(zI-D_H)^{-1},\qquad
 \rootC_H(z)=e_r^*\mathcal R_H(z)e_r
   =C_{\mu_H^{\mathrm{root}}}(z)=\frac{H'(z)}{rH(z)}.
\]
For \(z\notin Z(H)\cup Z(H')\), the compressed inverse satisfies
\begin{equation}\label{eq:appendix-compressed-resolvent}
 \begin{aligned}
 (zI-B_H)^{-1}
 &=\left(\mathcal R_H(z)-
       \frac{\mathcal R_H(z)e_re_r^*\mathcal R_H(z)}{\rootC_H(z)}\right)
       \big|_{e_r^\perp},\\
 \|(zI-B_H)^{-1}\|&\le\delta_H(z)^{-1}
       +\frac{\delta_H(z)^{-2}}{|\rootC_H(z)|},
 \end{aligned}
\end{equation}
where \(\delta_H(z)=\dist(z,Z(H))\). Indeed, the operator in parentheses annihilates \(e_r\), has range in \(e_r^\perp\), and multiplication by \(Q_r(zI-D_H)\) gives the identity there. The norm bound follows from \(\|\mathcal R_H(z)\|=\delta_H(z)^{-1}\).

Now let \(N=\deg H\ge1\) and define
\begin{equation}\label{eq:appendix-pooled-matrix}
 A(H)=D_H\oplus\bigoplus_{m=0}^{N-2}B_{H^{(m)}},\qquad
 M_N=\dim A(H)=\frac{N(N+1)}2,
\end{equation}
with the sum empty for \(N=1\). Each \(B_{H^{(m)}}\) is formed from the normal root matrix \(D_{H^{(m)}}\), not from a previously compressed matrix. Its characteristic polynomial and normalized eigenvalue-counting measure are
\[
 \det(zI-A(H))=\prod_{m=0}^{N-1}\widehat{H^{(m)}}(z),\qquad
 \pi_H:=\frac1{M_N}\sum_{m=0}^{N-1}\divz H^{(m)}
\]
respectively. Gauss--Lucas and compression give \(\|A(H)\|\le\max_{a\in Z(H)}|a|\). In particular, the normalized eigenvalue-counting measure of \(A_{P,n}:=A(P^n)\) is \(\overline\mu_{P,n}^{\mathrm{root}}\), and \(A_P:=\bigoplus_{n\ge1}A_{P,n}\) is bounded on the Hilbert direct sum of the finite block spaces. Explicitly, if \(\mathscr H_{P,n}\) is the space of the \(n\)th block, this sum consists of vectors \((v_n)\) with \(\sum_n\|v_n\|^2<\infty\), has norm \(\|(v_n)\|^2=\sum_n\|v_n\|^2\), and \(A_P(v_n)=(A_{P,n}v_n)\). For a bounded operator \(T\), write \(\sigma(T)\) for its spectrum and \(\sigma_{\mathrm{ess}}(T)\) for its Fredholm essential spectrum, the set of \(z\) for which \(zI-T\) is not Fredholm. Here Fredholm means that the range is closed, the kernel is finite-dimensional, and the range has finite codimension. We write \(\sigma_{\mathrm{disc}}(T)\) for the isolated spectral points whose associated Riesz (spectral) projections have finite rank. At an isolated point \(\lambda\), the projection is \((2\pi i)^{-1}\int_{|z-\lambda|=\varepsilon}(zI-T)^{-1}\dd z\), with a positively oriented circle enclosing no other spectral point.

\begin{theorem}\label{thm:appendix-differentiator-spectrum}
For every nonconstant polynomial \(P\),
\begin{equation}\label{eq:appendix-spectral-identities}
 \sigma(A_P)=\AllSh_P,\qquad
 \sigma_{\mathrm{ess}}(A_P)=\Sh_P,\qquad
 \sigma_{\mathrm{disc}}(A_P)=T_P.
\end{equation}
For every compact \(E\subset\C\setminus\Sh_P\) there are \(n_E\) and \(C_E<\infty\) such that
\begin{equation}\label{eq:appendix-tail-resolvent}
 \sup_{n\ge n_E}\sup_{z\in E}\|(zI-A_{P,n})^{-1}\|\le C_E.
\end{equation}
\end{theorem}
\begin{proof}
If \(P\) has only one distinct root, all blocks are scalar multiples of the identity and the result is immediate. Otherwise Theorem~\ref{thm:no-outliers} separates all descendant zero sets at sufficiently large levels from \(E\) by a common distance \(\delta_E>0\). We claim that the quantities \(\rootC_H\) in \eqref{eq:appendix-compressed-resolvent}, for \(H=(P^n)^{(m)}\), \(0\le m\le dn-2\), are uniformly bounded away from zero on \(E\) at large levels. If not, choose \(n_j\to\infty\), \(z_j\to z\in E\), and \(m_j/n_j\to\alpha\in[0,d]\), set \(H_j=(P^{n_j})^{(m_j)}\), and suppose that \(\rootC_{H_j}(z_j)\to0\). The normalized root measures of \(H_j\) tend to
\[
 \rootmu_P\quad(\alpha=0),\qquad
 \mu_{\alpha,P}\quad(0<\alpha<d),\qquad
 \delta_{\centroid_P}\quad(\alpha=d).
\]
The interior limit is Theorem~\ref{thm:fixed}. At zero, the retained root factors carry asymptotically all the mass; at \(d\), use Theorem~\ref{thm:endpoint-zeros}, which concentrates the entire root carrier at the centroid. Since \(\rootC_{H_j}=C_{\mu_{H_j}^{\mathrm{root}}}\), separation from \(E\) permits passage to Cauchy transforms at the moving points \(z_j\). Indeed, truncate \((z-a)^{-1}\) continuously inside \(|z-a|<\delta_E\): the truncation changes none of these integrals, and its boundedness and uniform continuity justify the simultaneous weak and moving-point limit. The limiting values are respectively \(P'(z)/(dP(z))\), \(C_{\mu_{\alpha,P}}(z)\), and \((z-\centroid_P)^{-1}\). They are nonzero: roots, critical points, and the centroid lie in \(\Sh_P\), and Proposition~\ref{prop:selected-C} treats the interior ratios. This contradiction proves the claim. Equation \eqref{eq:appendix-compressed-resolvent} now bounds all compression resolvents, while the diagonal blocks have bound \(\delta_E^{-1}\). This proves \eqref{eq:appendix-tail-resolvent}.

At \(z\notin\Sh_P\), the uniformly bounded block inverses in \eqref{eq:appendix-tail-resolvent} give an invertible tail for \(zI-A_P\). The remaining initial part is finite-dimensional, so \(zI-A_P\) is Fredholm of index zero. Off \(\AllSh_P\), the initial part is also invertible; conversely, closedness of the spectrum includes the closure of all finite-block eigenvalues. This proves the full-spectrum identity. For \(z\in\Sh_P\), Hausdorff convergence supplies eigenvalues \(z_n\in\sigma(A_{P,n})\) tending to \(z\). Their unit eigenvectors, placed in distinct blocks, have pairwise distance at least one and satisfy \(\|(A_P-zI)x_n\|=|z_n-z|\to0\). A Fredholm operator has a continuous quasi-inverse modulo a finite-rank map, so every bounded sequence with images tending to zero has a convergent subsequence. This contradicts the separation of these eigenvectors and proves the essential-spectrum identity. Finally Theorem~\ref{thm:literal} makes every remaining spectral point isolated and present in only finitely many blocks. Choose a closed isolating disk disjoint from the core. The tail estimate puts this disk in every late block's resolvent set, so coordinate evaluation of the Riesz contour integral vanishes in those blocks. Its range is thus finite-dimensional. Conversely, an isolated spectral point with a finite-rank Riesz projection is outside the Fredholm essential spectrum: the Cauchy-weighted resolvent integral is a quasi-inverse modulo that projection.
\end{proof}

The resolvent estimate is not automatic from convergence of zero sets: the bounded direct sum of nilpotent shifts of all sizes has spectrum the closed unit disk, although every finite spectrum is \(\{0\}\). Thus the theorem uses more than a diagonal listing of the zeros. It is nevertheless a consequence of our zero and transform estimates, not an independent proof of the geometry of the core.

\paragraph{Mixed traces.}
There is also a statement about the matrices and their adjoints. For every polynomial \(q\) in two noncommuting variables,
\begin{equation}\label{eq:appendix-mixed-traces}
 \lim_{n\to\infty}\frac1{\mathcal N_{d,n}}
       \operatorname{Tr}q(A_{P,n},A_{P,n}^*)
       =\int q(z,\bar z)\dd\rho_P(z).
\end{equation}
Here \(q\) is a finite linear combination of words in two letters, evaluated at \(A_{P,n}\) and \(A_{P,n}^*\); on the right it is evaluated at commuting scalars. In fact, for any nonconstant polynomial \(H\) of degree \(N\) with roots in \(|z|\le R\),
\[
 \left|\frac1{M_N}\operatorname{Tr}q(A(H),A(H)^*)
             -\int q(z,\bar z)\dd\pi_H(z)\right|
       \le\frac{C_{q,R}}N.
\]
For \(q=\sum_w c_w w\) one may take \(C_{q,R}=2\sum_w |c_w|(|w|+2)R^{|w|}\), where \(|w|\) is the word length. To see this, first let \(q\) be a word of length \(k\ge1\) and take \(R\ge1\). The full-space trace representing a compressed word has \(k\) factors \(Q_r\) after cyclicity of trace is used. Deleting these projections by telescoping changes the trace by at most \(kR^k\): each difference contains the rank-one factor \(I-Q_r\), and all remaining factors have the stated norm bound. Thus replacing every compression by its normal predecessor changes the pooled trace by \(O_{q,R}(N)\). Compared with the desired pooled root sum, the normal predecessor sum has one extra order-zero term and lacks the final degree-one term, another \(O_{q,R}(N)\) error. Constants cause no error in the normalized assertion. Divide by \(M_N\) and use Theorem~\ref{thm:averaging} to obtain \eqref{eq:appendix-mixed-traces}.

This is convergence in \(*\)-distribution; that is, convergence of all normalized mixed moments in the matrices and their adjoints, to a normal variable with law \(\rho_P\). The essential/discrete split is not a continuous/point split: every root of \(P\) is already an eigenvalue of infinite multiplicity of \(A_P\). Nor is \(\rho_P\) asserted to be a vector spectral measure of this generally nonnormal Hilbert-space operator.

\end{document}